\title{Block extensions, local categories, and basic Morita equivalences}

\author{
          TIBERIU COCONE\c T \\
\emph{Faculty of Economics and Business Administration}\\
\emph{Babe\c s-Bolyai University}\\
\emph{ Str. Teodor Mihali, nr.58-60}\\
\emph{ 400591 Cluj-Napoca, Romania}\\
\emph{tiberiu.coconet@math.ubbcluj.ro}\\
         ANDREI MARCUS\\
\emph{Faculty of Mathematics and Computer Science}\\
\emph{Babe\c s-Bolyai University}\\
\emph{Cluj-Napoca, Romania}\\
\emph{andrei.marcus@math.ubbcluj.ro}\\
          CONSTANTIN-COSMIN TODEA\footnote{Corresponding author} \\
 \emph{Technical University of Cluj-Napoca}\\
 \emph{Department of Mathematics, Str. G. Baritiu 25}\\
 \emph{Cluj-Napoca 400027, Romania}\\
 \emph{Simion Stoilow Institute of Mathematics of the Romanian Academy}\\
 \emph{ P.O. Box 1-764,014700}\\
 \emph{ Bucharest, Romania}\\
 \emph{Constantin.Todea@math.utcluj.ro}
}
\date{09.05.2019}

\documentclass[10pt,fleqn]{article}
\usepackage{xypic}
\usepackage{amsmath,amsthm,amssymb,amsfonts}
\usepackage[a4paper]{geometry}
\newcommand{\Hom}{\operatorname{Hom}}

\newcommand{\Res}{\operatorname{Res}}
\newcommand{\Tr}{\operatorname{Tr}}
\newcommand{\End}{\operatorname{End}}
\newcommand{\Ind}{\operatorname{Ind}}
\newcommand{\Aut}{\operatorname{Aut}}
\newcommand{\Out}{\operatorname{Out}}

\newcommand{\Ker}{\operatorname{Ker}}

\newcommand{\Inj}{\operatorname{Inj}}
\newcommand{\Int}{\operatorname{Int}}
\newcommand{\hU}{\operatorname{hU}}

\newcommand{\Br}{\operatorname{Br}}

\newcommand{\Od}{\mathcal{O}}
\newcommand{\Sum}{\displaystyle\sum}

\newtheorem{thm}{Theorem}[section]

\newtheorem{corollary}[thm]{Corollary}
\newtheorem{lemma}[thm]{Lemma}
\theoremstyle{definition}
\newtheorem{proposition}[thm]{Proposition}
\newtheorem{definition}[thm]{Definition}
\newtheorem{nim}[thm]{}
\newtheorem{remark}[thm]{Remark}

\begin{document}
\maketitle


\begin{abstract}
Let $(\mathcal{K},\mathcal{O},k)$ be a $p$-modular system with $k$ algebraically closed, let $b$ be a block of the normal subgroup $H$ of $G$ having defect pointed group $Q_\delta$ in $H$ and $P_\gamma$ in $G$, and consider the block extension $b\mathcal{O}G$. One may attach to $b$ an extended local category $\mathcal{E}_{(b,H,G)}$, a group  extension $L$ of $Z(Q)$ by $N_G(Q_\delta)/C_H(Q)$ having $P$ as a Sylow $p$-subgroup, and a cohomology class $[\alpha]\in H^2(N_G(Q_\delta)/QC_H(Q),k^\times)$. We prove that these objects are invariant under the $G/H$-graded basic Morita equivalences. Along the way, we give alternative proofs of the results of K\"ulshammer and Puig (1990), Puig and Zhou (2012) on extensions of nilpotent blocks. We also deduce by our methods a result of Zhou (2016) on $p'$-extensions of inertial blocks.
\end{abstract}

\section{Introduction}\label{sec1}

This paper is an effort to unify and simplify  results involving nilpotent or inertial blocks of normal subgroups and the basic Morita equivalences introduced by L.~Puig \cite{PuLo}. Recall that a block algebra is called inertial in \cite{Pu2} if it is basic Morita equivalent to its Brauer correspondent. Nilpotent blocks are the particular case when the inertial quotient is trivial. The structure of the source algebra of an extension of a nilpotent block has been determined by B.~K\"ulshammer and L.~Puig \cite[Theorem 1.12]{KuPu}, together with a strong statement of uniqueness of a group controlling the fusion    \cite[Theorem 1.8]{KuPu}. Their results significantly extend the result of \cite{Li}, and both theorems have been given simplified proofs in \cite[Theorems 3.14 and 3.5]{PuZhIII}. In a similar fashion, Y.~Zhou \cite{Zhou1} has determined the source algebra of a $p'$-extension of an inertial block.

We treat here these results in the framework of group graded basic Morita equivalences  introduced in \cite{CoMa}, and further investigated in \cite{CoMaTo}. To summarize our results, let us introduce some notation. Let $(\mathcal{K},\mathcal{O},k)$ be a $p$-modular system,  with $k=\mathcal{O}/J(\mathcal{O})$ algebraically closed (although this assumption is not needed everywhere).  Let $H$ be a normal subgroup of a finite group $G$, let $\bar G=G/H$, and let $b$ be a block of $\mathcal{O}H$, which may be assumed to be $G$-invariant. Let $Q_\delta$ be a defect pointed subgroup of $H_{\{b\}}$, and $P_\gamma$ a defect pointed subgroup of $G_{\{b\}}$, such that $Q_\delta\le P_\gamma$ and $Q=P\cap H$.

Let $b_\delta$ be the block of $\mathcal{O}C_H(Q)$ with defect group $Z(Q)$ determined by $Q_\delta$, so $N_G(Q_\delta)$ is the stabilizer of $b_\delta$ in $N_G(Q)$. We denote $E=E_G(Q_\delta)=N_G(Q_\delta)/C_H(Q)$ and $\tilde E=N_G(Q_\delta)/QC_H(Q)$. Recall that $b_\delta$ is also a block of $\mathcal{O}QC_H(Q)$ with defect group $Q$, and it has a unique simple module $\bar V$. Since $QC_H(Q)\trianglelefteq N_G(Q_\delta)$, the Clifford extension of $\bar V$ with respect to this situation gives a $2$-cocycle $\alpha\in Z^2(\tilde E,k^\times)$.

As in \cite[Theorem 1.8]{KuPu} and \cite[Theorem 3.5]{PuZhIII}, the block $b_\delta$ determines a group extension $L$ of $Z(Q)$ by $E$, having $P$ as a Sylow $p$-subgroup, such that the conjugation actions of $N_G(Q_\delta)$ on $Q$ and of $L$ on $Q$ are strongly related. We may therefore consider (with some  abuse of notation) the twisted group algebra $\mathcal{O}_\alpha L$.  Note that we do not assume here that the block $b$ is nilpotent.

\begin{thm} \label{thmA} Consider the block extension $A=b\mathcal{O}G$   and the $\bar G$-graded crossed product $A_\delta=jAj$, where $j\in \delta$. Assume that $B=b\mathcal{O}H$ is an inertial block. Then the following statements hold.

{\rm a)} The bimodule inducing the Morita equivalence between $B$ and $\mathcal{O}_\alpha(Q\rtimes \tilde E_H(Q_\delta))$ is $\bar G$-invariant.

{\rm b)}  {\rm(\cite[Theorem 1.12]{KuPu} and \cite[Theorem 3.14]{PuZhIII})} If $b$ is nilpotent (hence $\tilde E\simeq \bar G$), then there is a $\bar G$-graded basic Morita equivalence between $A$ and $\mathcal{O}_\alpha L$.

{\rm c)}  {\rm(\cite[Theorem]{Zhou1})} If $p$ does not divide the order of $\bar G$ (hence $Q_\delta=P_\gamma$),  then $A_\delta$ is an $\tilde E$-graded algebra with identity component $C_\delta$ Morita equivalent to $\mathcal{O}Q$, and there is a $\bar G$-graded basic Morita equivalence between $A$ and $\mathcal{O}_\alpha(Q\rtimes \tilde E_G(Q_\delta))$. Moreover the Clifford extension of the unique simple $C_\delta$-module is isomorphic to the Clifford extension of $\bar V$.
\end{thm}

We obtain the Morita equivalences in b) and c) by extending  a bimodule inducing the given Morita equivalence to a certain $\bar G$-graded diagonal subalgebra, as in \cite[Theorem 5.1.2]{Ma}. Note that in c), the original bimodule extends directly, while in b), the original bimodule does not extend in general, it must be replaced by another one.

We do not know whether there is a common generalization of b) and c). Zhou \cite{Zhou2} considered a particular case of a $p$-extension of an inertial block, but the Morita equivalence obtained there  is not basic.

The extended local category $\mathcal{E}_{(b,H,G)}$ of $(G,\bar G)$-fusions was introduced in \cite[3.4]{PuZhIII}. The objects are pointed subgroups of $P_\gamma$, and the morphisms are conjugations by elements $x\in G$, also taking into account the class $\bar x\in \bar G$.

Now let $b'\mathcal{O}G'$ be another block extension, where $H'\trianglelefteq G'$, $b'\in Z(\mathcal{O}H')$ is a $G'$-invariant block, and $G'/H'\simeq G/H=\bar G$. We use ``{ $'$ }" to denote the objects associated with $b'$.

\begin{thm} \label{thmB} Assume that there is a $\bar G$-graded basic Morita equivalence between $A=b\mathcal{O}G$ and $A'=b'\mathcal{O}G'$. Then, by identifying $P$ and $P'$, we have:

{\rm a)} The categories $\mathcal{E}_{(b,H,G)}$ and $\mathcal{E}_{(b',H',G')}$ are equivalent; in particular, $E\simeq E'$.

{\rm b)} The group extensions $L$ and $L'$ of $Z(Q)$ by $E$ are isomorphic.

{\rm c)} $[\alpha]=[\alpha']$ in $H^2(\tilde E, k^\times)$.
\end{thm}

In particular, \cite[Theorem 1.8]{KuPu} and \cite[Theorem 3.5]{PuZhIII} follow from Theorem \ref{thmB} and Theorem \ref{thmA}.d). Note that in these papers, the fact that $L$ controls the fusion was first established, and then used to get the Morita equivalence.

The paper is organized as follows. In Section \ref{sec-prelim} we recall the main concepts and facts needed for the proofs of the main results. In Section \ref{sec2} we give the construction of  the group extension $L$ of $Z(Q)$ by $E$, and of the twisted group algebra $\mathcal{O}_\alpha L$, by using the properties of the block extension $\mathcal{O}N_G(Q_\delta)b_\delta$. In Section  \ref{sec4} we prove a variant of the Fong-Reynolds reduction, which reduces us to the case when $b_\delta$ is $N_G(Q)$-invariant. In Section \ref{sec3}  we show that under a certain condition, there is an injective algebra map from the source algebra of the block extension $\mathcal{O}N_G(Q_\delta)b_\delta$ to the source algebra of $\mathcal{O}Gb$. In Section  \ref{sec5} we prove the useful Lemma \ref{lemaDade*} which allows to construct group graded Morita equivalences, based on some results of E.C.~Dade on extendibility of modules. This is applied in Section \ref{sec7}, where we prove the statements of Theorem \ref{thmA}. We discuss $(A,\bar G)$-fusions and $(G,\bar G)$-fusions in Section \ref{s:fusions}, and prove Theorem \ref{thmB} in the last section.

\section{Preliminaries and quoted results}\label{sec-prelim}

In this paper, all algebras and modules, considered as left modules, are finitely generated. We introduce in this section our notation and recall some basic concepts and facts. Our assumptions are standard, and we refer to \cite{The} and \cite{Ma}  for general results on block theory and group graded algebras, respectively.

\begin{nim} Throughout this paper $p$ is a prime, and let $k$ be the residue field, with characteristic $p$, of a discrete valuation ring $\mathcal{O}$. For the moment, we do not assume that $k$ is algebraically closed, some of the results in the next sections will require this assumption.

Let   $H$ be a normal subgroup of a finite group $G$.  Then the group algebra $\Od H$ is a $p$-permutation $G$-algebra. Let $b$ be a $G$-invariant block of $\mathcal{O}H$, hence, in particular, $b$ remains a primitive idempotent in $(\Od H)^G$.  Set the notations
\[\bar G:=G/H, \qquad A:=\Od Gb, \qquad B:=\Od Hb,\]
and we usually regard $A$ as a $\bar G$-graded $\mathcal{O}$-algebra with identity component $B$.
\end{nim}

\begin{nim} \label{n:source} Take a defect pointed group $P_{\gamma}$  of $G_{\{b\}}$. Then, by \cite[Proposition 5.3]{KuPu}, there is  a defect pointed group $Q_{\delta}$ of $H_{\{b\}}$ on the $H$-interior algebra $\Od H$ such that $Q_\delta\leq P_{\gamma}$;  in this case we may assume that $Q=P\cap H$.

The Frattini argument implies that $G=HN_G(Q_{\delta})$, hence \[\bar G \simeq N_G(Q_{\delta})/N_H(Q_{\delta}).\] Let $i\in \gamma$ and $j\in\delta$ be source idempotents such that $j=ij=ji$. Set
\[A_{\gamma}=iAi, \qquad A_{\delta}=jAj, \qquad B_{\delta}=jBj, \qquad B_{\gamma}=iBi.\]
By \cite[Proposition 3.2]{Ma2}, both $A_\delta$ and $A_\gamma$ are strongly $\bar G$-graded $\mathcal{O}$-algebras, and we have $\bar G$-graded Morita equivalences between $A$, $A_\delta$ and $A_\gamma$.
\end{nim}

\begin{nim} Recall (see \cite[\S\ 40]{The}) that there is a maximal $(H,b)$-Brauer pair, denoted $(Q,b_{\delta})$, associated with $Q_{\delta}$. Here $b_{\delta}$ is a block of $kC_H(Q)$ with defect group $Z(Q)$. It is well known that $b_{\delta}$ lifts uniquely to a block (still denoted by $b_{\delta}$) of $\mathcal{O}C_H(Q)$ with defect group $Z(Q)$, and that we have the equality $N_H(Q_{\delta})=N_H(Q,b_{\delta})$. Moreover, $b_{\delta}$ is also a block of $\mathcal{O}QC_H(Q)$ with defect group $Q$. As in \cite[Section 4]{CoMaTo}, we also denote
\[E:=E_G^{\bar G}(Q_\delta)=N_G(Q_{\delta})/C_H(Q), \quad \tilde{E}:=\tilde{E}^{\bar G}_G(Q_\delta) =N_G(Q_{\delta})/QC_H(Q),\]
\[E_H(Q_{\delta})=N_H(Q_{\delta})/C_H(Q), \qquad \tilde{E}_H(Q_{\delta})=N_H(Q_{\delta})/QC_H(Q).\]
These groups will be regarded in Section  \ref{s:fusions} as automorphism groups in a certain extended local category.
\end{nim}

\begin{nim} As we deal with Morita equivalences, we consider another finite group $G'$, and assume that $\omega:G\rightarrow \bar{G}$, $\omega':G'\rightarrow \bar{G}$ are group epimorphisms with $H=\Ker \omega$ and $H'=\Ker \omega'$. Let  $b'$  be a $G'$-invariant block of $\Od H'$   and let $A':=\Od G'b'$ be the associated $\bar{G}$-graded, $G'$-interior algebra with $1$-component $B':=A'_1=\Od H'b'$.
We will use notations similar to the above (and self-explanatory) for the objects associated with $b'$.

Set $\ddot{G}:=(\omega\times\omega')^{-1}(\Delta(\bar{G}))$  and $\ddot{\Delta}:=(b\otimes b')\Od\ddot{G}$.
\end{nim}

\begin{nim} Basic Morita equivalences have been introduced by Puig \cite{PuLo}. Recall that if the indecomposable $\Od(H\times H')$-module $M$, such that $bMb'=M$, induces a  Morita equivalence between $\Od Hb$ and $\Od H'b'$, then  by \cite[Theorem 6.9]{PuLo} there is a $\ddot{Q}$-interior algebra embedding $B_{\delta}\rightarrow T\otimes B'_{\delta '}$, where $T=\End_{\Od}(N)$ is a ${Q}$-interior algebra, with $\ddot{Q}\leq H\times H'$ a vertex of $M$, and  $N$ a source $\Od\ddot{Q}$-module of $M$. The equivalence is called {\it basic} if $T$ is a Dade $\ddot{Q}$-algebra (see \cite[Chapter 7]{PuLo}). In this case, the projections from $\ddot{G}$ to $G$ and $G'$ give an identification of the vertex $\ddot{Q}$ with the defect groups $Q$ and $Q'$.

The block $b$ is called \textit{inertial} (see \cite[2.16]{Pu2}) if it is basic Morita equivalent to the block  $\Od N_H(Q_{\delta})b_{\delta}$ (hence to its Brauer correspondent as well).
\end{nim}

\begin{nim} \label{subsec82}  Basic Morita equivalences have been generalized to block extensions in \cite{CoMa}. According to \cite[Definition 4.2]{CoMa}, $A$ is $\bar{G}$-{\it graded basic} Morita equivalent to $A'$ if and only if there is  an indecomposable $\Od[H\times H']$-module $M$ associated with $b\otimes b'$ that extends to $\ddot{\Delta}$ such that there is an  embedding $f:A_{\gamma}\rightarrow S\otimes A'_{\gamma'}$ of $\bar G$-graded $P$-interior algebras, where $S:=\End_{\Od}(\ddot{N})$ is a Dade $\ddot{P}$-algebra, with $\ddot{P}$ a vertex of $M$, and $\ddot{N}$ a source $\Od\ddot{P}$-module of $M$.

In this case, we can again identify
\begin{equation}\label{eq1PP'} P\simeq P'\simeq \ddot{P}
\end{equation}
In particular, it follows  that the embedding $f$ restricts to a $Q$-interior algebra embedding
\begin{equation}\label{eq2emb}  f:B_{\gamma}\rightarrow S\otimes B'_{\gamma'}
\end{equation}
Let $R\leq P$ and $R'\leq P'$ such that $R'$ corresponds to $R$ via (\ref{eq1PP'}). By applying the Brauer construction to (\ref{eq2emb}) we obtain the $N_P(R)$-algebra embedding
\begin{equation}\label{eq3emb}
\bar{f}: B_{\gamma}(R)\rightarrow S(R)\otimes B'_{\gamma'}(R')
\end{equation}
If $R_{\epsilon}$ is a local pointed subgroup of $P_{\gamma}$ then, by (\ref{eq3emb}), there is a unique local pointed group $R'_{\epsilon'}$  such that $R'_{\epsilon'}\leq P'_{\gamma'}$. This construction yields a bijection
\begin{equation}\label{eq9} R_{\epsilon}\leftrightarrow R'_{\epsilon'}
\end{equation}
between the local pointed groups included in $P_{\gamma}$ and local pointed groups included in $P'_{\gamma'}$.
\end{nim}

\begin{nim}  \label{ss:Clifford-ext} The following remarks hold, more generally, when $A$ is a strongly $\bar G$-graded $\mathcal{O}$-algebra with $1$-component $B:=A_1$, see \cite[\S 1]{D84} or \cite[Section 2.2]{Ma}.  In the sequel, we mostly apply this in the situation when $A$ is a crossed product of $B$ and $\bar G$, which means that the group $\mathrm{hU}(A)=\bigcup_{g \in\bar G}(A^\times\cap A_g)$ of homogeneous units of $A$ is a group extension of $B^\times$ by $\bar G$.

Let $V$ be an $A_1$-module. For $g\in \bar G$, the $A_1$-module $A_g\otimes_{B}V$ is called the $g$-conjugate of $V$, and $V$ is called $\bar G$-invariant, if $V\simeq A_g\otimes_{B}V$ as $B$-modules for all $g\in \bar G$. Consider the $\bar G$-graded $A$-module $U=A\otimes_{B}V$. Then the endomorphism algebra $A':=\End_A(U)^{\mathrm{op}}$ is a $\bar G$-graded algebra with  $g$-component  given by
\[A'_g\simeq \Hom_{A_1}(V,A_g\otimes_{B}V)\]
for any $g\in \bar G$. In this way, $U$ becomes a $\bar G$-graded $(A,A')$-bimodule.

Note that $A'$ is a crossed product of $B'\simeq\End_B(V)^{\mathrm{op}}$ and $\bar G$ if and only if $V$ is a $\bar G$-invariant $B$-module. Then the graded Jacobson radical $J_{\mathrm{gr}}(A')$ equals $J(B')A'=A'J(B')$, and $\bar{A}':=A'/J_{\mathrm{gr}}(A')$ is still a crossed product of $\bar B':=B'/J(B')$ and $\bar G$. In this case, the group extension $\mathrm{hU}(A')$ is called the {\it Clifford extension} of $V$, while $\mathrm{hU}(\bar{A}')$ is the {\it residual Clifford extension} of $V$.

If, in addition, $k$ is algebraically closed, then $\bar A'$ is a twisted group algebra of the form $k_\alpha\bar G$, for some 2-cocycle $\alpha\in Z^2(\bar G, k^\times)$.
\end{nim}

\begin{nim}  \label{ss:grade-refine} We will need in Section \ref{sec7} the following observation which allows to refine the grading of an algebra.

Let $A$ be a $\bar G$-graded crossed product with $1$-component $B$, and assume that $B$ is a $H$-graded crossed product with $1$-component $C$. Assume also that for any $x\in G$ there is an invertible element $a_x\in A_{\bar x}$ such that the following conditions hold.
\begin{enumerate}
\item[(1)] $a_h\in B_h$ for all $h\in H$.
\item[(2)] $a_xCa_x^{-1}=C$ for all $x\in G$.
\item[(3)] $a_xa_ya_{xy}^{-1}\in C^\times$ for all $x,y\in G$.
\end{enumerate}
Then the $\bar G$-grading of $A$ can be refined to a $G$-grading such that $A_h=B_h$ for all $h\in H$.

Indeed, letting $A_x=Ca_x=a_xC$ for all $x\in G$, it is easy to see that $A=\bigoplus_{x\in G}A_x$ is a $G$-graded crossed product.
\end{nim}

\section{The block $b_\delta$ and the extension $L$ of $Q$} \label{sec2}

\begin{nim} We keep the setting of Section 1. We consider the $N_G(Q_{\delta})$-invariant nilpotent block $b_{\delta}$ of $\mathcal{O}QC_H(Q)$. In this case $(Q,b_{\delta})$ remains a maximal $(b_{\delta},QC_H(Q))$-Brauer pair. We obtain as a particular case of the next proposition that $P$ is also a defect of  $b_{\delta}$ in $N_G(Q_{\delta})$. For this, note that associated with $P_{\gamma}$, there is a so-called generalized maximal $(G,H,b)$-Brauer pair denoted $(P,b_{\gamma})$ such that $(Q,b_{\delta})\leq(P,b_{\gamma})$. In particular, if $G/H$ is a $p'$-group, then $Q=P$ and $b_{\gamma}=b_{\delta}$. We refer to \cite{KeSt} for more details regarding generalized Brauer pairs.
\end{nim}

The next proposition is a generalization of \cite[Proposition 3.1]{To}, but we will use it here only in the case $R=Q$ and $b_{\epsilon}=b_{\delta}$.

\begin{proposition} \label{propPdefectbdelta} Let $R_{\epsilon}\leq P_{\gamma}$ be local pointed groups on $\mathcal{O}H$, where $R$ is a normal subgroup of $P$ such that $R\leq H$ and $(R,b_{\epsilon})\leq (P,b_{\gamma})$. Assume that $P\le N_G(R_\epsilon)$. Then the pair $(P,b_{\gamma})$ is a maximal $(N_G(R_{\epsilon}), N_H(R_{\epsilon}), b_{\epsilon})$-Brauer pair.
\end{proposition}

\begin{proof}
Let  $X=N_G(P)\cap N_G(R_{\epsilon}).$ The following restriction of the Brauer homomorphism
\[\Br_P: (kC_H(R))^{N_G(R_{\epsilon})}_P\to (kC_H(P))^X\] is an epimorphism of $N_G(P)$-algebras. Indeed, we have
\begin{align*}\Br_P\left(\Tr_P^{N_G(R_{\epsilon})}(a)\right)&=\Br_P\left(\sum_{x\in [X\setminus N_G(R_{\epsilon})/P]} \Tr_{X\cap P^x}^{X}(a^x)\right)\\
       &=\sum_{x\in [X/P]}\Br_P(a^x)=\Tr_P^X(\Br_P(a)),
\end{align*}
where $a\in (kC_H(R)^P$. Note that $x\in X$ if and only if $X\cap P^x=P,$ and then, if $x\notin X$, we get
\[\Tr_{X\cap P^x}^X(a^x)=\sum_{y\in [P\setminus X/(X\cap P^x)]}\Tr_{P\cap (X\cap P^x)^y}^{P}(a^{xy}),\] where $P\cap (X\cap P^x)^y$ is always a proper subgroup of $P.$

By using the inclusion
\begin{align*} (kC_H(P))^X_P=(kC_H(P))^X &\subseteq (kC_H(P))_P^{N_G(P_{\gamma})}= (kC_H(P))^{N_G(P_{\gamma})}
\end{align*}
and the fact that $b_{\gamma}$ is a primitive idempotent of  the algebra $(kC_H(P))^{N_G(P_{\gamma})}$, we can find a primitive idempotent $b_X\in (kC_H(P))^X$ that verifies
\[b_Xb_{\gamma}=b_{\gamma}b_X=b_{\gamma}.\]
By lifting it, we obtain the equality $\Br_P(\bar{b}_{\epsilon})=b_X,$ for a primitive idempotent $\bar{b}_{\epsilon}$ with defect group $P$ in the algebra $(kC_H(R))^{N_G(R_{\epsilon})}_P$.

Further, we have
\[b_{\gamma}=b_{\gamma}\Br_P(b_{\epsilon})=b_Xb_{\gamma}\Br_P(b_{\epsilon})=b_{\gamma}\Br_{P}(\bar{b}_{\delta})\Br_P(b_{\epsilon}),\]
hence $\bar{b}_{\epsilon}b_{\epsilon}\neq 0.$ This forces
\[\bar{b}_{\epsilon}b_{\epsilon}= \bar{b}_{\epsilon}=b_{\epsilon},\] since both idempotents are primitive in
$(kC_H(R))^{N_G(R_{\epsilon})}.$
\end{proof}

\begin{nim} By \cite[Proposition 5.3]{KuPu} we have that $PC_H(Q)/C_H(Q)\simeq P/Z(Q)$ is a Sylow $p$-subgroup of $E=N_G(Q_\delta)/C_H(Q)$. For now, we only need to show the existence of  an extension $L$ of $Q$ by $\tilde E =N_G(Q_\delta)/QC_H(Q)$ containing $P$ as a Sylow $p$-subgroup. It turns out that when $b$ is nilpotent, $L$ controls the fusions in the block extension $b\mathcal{O}G$ (\cite[Theorem 1.8]{KuPu}), but this will be obtained as a consequence of a more general result in Section \ref{s:fusions}. The group $L$ exists without any assumption on $b$. For convenience, we include the proof, which follows the first part of the proof of \cite[Theorem 3.5]{PuZhIII}.
\end{nim}

\begin{proposition} With the above notations there is a group extension
\[\xymatrix{1\ar[r]& Q\ar[r]^{\tau}&  L  \ar^{\bar\pi}[r] &  \tilde{E} \ar[r]&1}\]
and an injective group homomorphism $\tau:P\to L$ such that $\tau(P)$ a Sylow $p$-subgroup in $L$, $\ker\bar \pi=\tau(Q)$
and $(\pi\circ\tau)(u)=\bar\pi(u)$ for any $u\in P$.
\end{proposition}

\begin{proof} Consider the element $\bar h\in H^2(PC_H(Q)/C_H(Q), Z(Q))$ corresponding to the extension
\[1\to Z(Q) \to P \to P/Z(Q) \to 1.\]
We show that $\bar h$ is in the image of the restriction map $H^2(N_G(Q_\delta)/C_H(Q),Z(Q)) \to H^2(PC_H(Q)/C_H(Q),Z(Q)).$ By \cite[Chapter XII, Theorem 10.1]{CE}, it is enough to prove that for any subgroup $R$ of $P$ containing $Q$, and any  $\bar x\in N_G(Q_\delta)/C_H(Q)$ such that $RC_H(Q)\le P^xC_H(Q)$, the restriction to $RC_H(Q)/C_H(Q)$ of the class from $H^2(P^xC_H(Q)/C_H(Q),Z(Q))$ determined by $P^x$ coincides with the class determined by $R$.

Indeed, by Proposition \ref{propPdefectbdelta} applied for $Q$, we get that $b_\delta$ is a block idempotent of $\mathcal{O}PC_H(Q)$, of $\mathcal{O}P^xC_H(Q)$ and of $\mathcal{O}RC_H(Q)$ having defect groups $P$, $P^x$ and $R$, respectively. Since $RC_H(Q)\le P^xC_H(Q)$, $R$ must be contained in a conjugate of $P^x$, so there is $z\in C_H(Q)$ such that $R\le P^{zx}$. This implies that the extension
\[1\to Z(Q) \to R \to R/Z(Q) \to 1\]
is a subextension of
\[1\to Z(Q) \to P^x \to P^x/Z(Q) \to 1,\]
and the claim follows.

Therefore, we obtain a group extension
\[\xymatrix{1\ar[r]& Z(Q)\ar[r]^{\tau} & L\ar[r]^{\pi} & E \ar[r]&1}\]
corresponding to the element in $H^2(N_G(Q_\delta)/C_H(Q),Z(Q))$ whose image by restriction in $H^2(PC_H(Q)/C_H(Q),Z(Q))$ is $\bar h$, and also the injective group extension map $\tau:P\to L$.
\end{proof}

\begin{remark} \label{r:conjQ} In the situation of the above proposition, it is clear by the construction of the group $L$ that for any $x\in N_G(Q_\delta)$ there is $y\in L$ such that \[y\tau(u)y^{-1} = \tau(xux^{-1})\] for all $u\in Q$.
\end{remark}

\begin{nim} \label{n:Lpprime} We consider here the case when $\bar G$ is a $p'$-group, which is the assumption in Subsection \ref{ssec4} below. It is known that when $k$ is algebraically closed, $\tilde E_H(Q_{\delta}):=N_H(Q_{\delta})/QC_H(Q)$ is a $p'$-group, and therefore $\tilde E$ is still a $p'$-group, since $\tilde E/\tilde E_H(Q_{\delta})\simeq G/H.$ We have the  commutative diagram
\begin{equation*}\label{eqtheta}
\qquad\xymatrix{ 1\ar[r] & Q/Z(Q)\ar[r]\ar[d] & N_G(Q_{\delta})/C_H(Q)\ar[r]\ar[d] & \tilde{E} \ar[d]\ar@{.>}[dl]^{\theta} \ar[r]&1\\
                                            1\ar[r]&\mathrm{Int} Q\ar[r]&\Aut Q\ar[r]
                                            &\Out Q\ar[r]&1. }
\end{equation*}
Since $\tilde{E}$ is a $p'$-group we obtain an action of $\tilde{E}$ on $\Od Q$. In this case, we have that $L\simeq Q\rtimes_{\theta}\tilde{E}$. Note also that we have the group isomorphism $\tilde{E}\simeq L/Q.$
\end{nim}

\begin{nim} \label{d:cocycle-beta} Assume that $k$ is algebraically closed. Let $\bar V$ be the unique simple $\mathcal{O}QC_H(Q)b_\delta$-module. The Clifford extension of $\bar V$ (see  \ref{ss:Clifford-ext}) gives a $2$-cocycle $\alpha \in Z^2(\tilde{E}, k^\times)$. Since the group extension
\[1\to 1+J(\mathcal{O})\to \mathcal{O}^\times\to k^\times \to 1\]
splits uniquely, we obtain a $2$-cocycle in $Z^2(\tilde{E}, \mathcal{O}^\times)$. From the isomorphism $L/Q\simeq \tilde{E}$, we obtain by inflation a $2$-cocycle, also denoted by $\alpha$, in $Z^2(L, \mathcal{O}^\times)$. Using this, we construct the twisted group algebra $\mathcal{O}_\alpha L$, which will be regarded as an $\tilde{E}$-graded algebra with $1$-component $\mathcal{O}Q$, or even as an $E$-graded algebra with $1$-component $\mathcal{O}Z(Q)$. The restrictions of $\alpha$ to subgroups of $L$ will be denoted by the same $\alpha$.
\end{nim}

\section{The Fong-Reynolds correspondence}\label{sec4}

\begin{nim} It is clear that  $b_{\delta}$ remains a  primitive idempotent in $(\Od N_H(Q_{\delta}))^{N_G(Q_{\delta})}\ $, and the induced block $b':=b_{\delta}^{N_H(Q)}$ is the Brauer correspondent of $b$.
 Since $b$ is a $G$-invariant block of $\mathcal{O}H$ with defect $Q$, it is well known that $b'$  is an $N_G(Q)$-invariant block of $\Od N_H(Q)$ with defect group $Q$ in $N_H(Q)$. In this section, set
\[G':=N_G(Q), \qquad H':=N_H(Q), \qquad A':=\Od N_G(Q)b', \qquad B':=\Od N_H(Q)b'.\]
Since  \[G/H\simeq N_G(Q_{\delta})/N_H(Q_{\delta})= G'/H',\] it follows that $A'$  is also a $\bar G$-graded algebra, with identity component $B'$.
\end{nim}

The Fong-Reynolds theorem says that there is a Morita equivalence between the block algebras $B=\mathcal{O}N_H(Q_{\delta})b_{\delta}$ and $B'=\mathcal{O}N_H(Q)b'$. This equivalence is actually basic. Although the argument is well known, for completeness, we give an explicit proof of  the fact that Fong-Reynolds equivalence extends to group-graded Morita equivalences.

We start with a useful lemma, whose proof is actually hidden in the proof of \cite[Proposition 3.2]{To}, see \cite[Remarks 3.3, 3.4]{To}.

\begin{lemma}\label{propblocskdefect} With the above notations, the  Brauer correspondent $b'$ of $b$ in $\mathcal{O}N_H(Q)$ is a primitive idempotent of $(\mathcal{O}N_H(Q))^{N_G(Q)}$, and  moreover, $b'=\Tr_{N_G(Q_{\delta})}^{N_G(Q)}(b_{\delta})$.

\end{lemma}

\begin{proof} Since  $b'$ is the induced block of $b_{\delta}$ from $N_H(Q_{\delta})$ to $N_H(Q)$, we have that
\begin{equation} \label{eq1'}
\Br_Q(c)b_{\delta}=b_{\delta}.
\end{equation}
By \cite[Lemma 2.1]{To} it is clear that $c$ is a primitive idempotent of $(\mathcal{O}C_H(Q))^{N_G(Q)}$ and then from \cite[Proposition 3.10, Part IV]{AKO} we deduce that $c=\Tr_{N_G(Q,b_1)}^{N_G(Q)}(b_1)$, where $b_1$ is a block of $\mathcal{O}C_H(Q)$.
It is enough to show that
\begin{equation*}\label{eq3c'}
b_{\delta}={}^gb_1
\end{equation*}
for some $g\in G$. By using (\ref{eq1'}), we deduce the equalities
\[\Br_{Q}(c)b_{\delta}=\sum_{g\in[N_G(Q)/N_G(Q,b_1)]}\Br_{Q}({}^gb_1)b_{\delta}=b_{\delta}.\]
Since for all elements $g\in[N_G(Q)/N_G(Q,b_1)]$ the idempotents ${}^gb_1$ are blocks of $\mathcal{O}C_H(Q)$, it follows that
\[b_{\delta}=\sum_{g\in[N_G(Q)/N_G(Q,b_1)]}{}^gb_1b_{\delta},\]
hence there is a unique $g\in[N_G(Q)/N_G(Q,b_1)]$ such that $b_{\delta}={}^gb_1.$
\end{proof}

The main result of this section is the following refinement of the Fong-Reynolds correspondence.

\begin{proposition}\label{propFong} With the above notations,  there is a $\bar G$-graded basic Morita equivalence between $\mathcal{O}N_G(Q_{\delta})b_{\delta}$ and $\mathcal{O}N_G(Q)b'$;
\end{proposition}

\begin{proof} First notice that by Lemma \ref{propblocskdefect}  we have $b'=\Tr_{N_G(Q_{\delta})}^{N_G(Q)}(b_{\delta}).$  Next, by \cite[Lemma 2.3.16]{Ma} we know that there is a $\bar G$-graded Morita equivalence between $b_{\delta}\Od N_G(Q)b_{\delta}$ and $\Od N_G(Q)b_{\delta}\Od N_G(Q)$ induced by $\Od N_G(Q)b_{\delta}$ and its $\mathcal{O}$-dual $b_{\delta}\Od N_G(Q)$, viewed as $\bar G$-graded bimodules.

The identity component $\Od N_H(Q)b_{\delta}$ of $\Od N_G(Q)b_{\delta}$ determines a Morita equivalence between $\Od N_H(Q)b'$ and $\Od N_H(Q_{\delta})b_{\delta}$. Moreover, we have that $\Od N_H(Q)b_{\delta}$ is an indecomposable $\Od(N_H(Q)\times N_H(Q_{\delta}))$-module that extends to the diagonal subalgebra, and we claim that this extension has vertex $\Delta(P\times P)$. Indeed, if $\ddot{P}\leq G\times G'$ denotes such a vertex of $\Od N_H(Q)b_{\delta}$, we know that the projection $G\times G'\to G$ restricts to the epimorphism $\ddot{P}\to P,$ since $P$ is a defect group in $N_G(Q_{\delta})$ of $b_{\delta}.$ Since $b_{\delta}$ is projective relative to $P$ it follows that $\Od N_H(Q)b_{\delta}$ is relatively $\Delta(P\times P)$-projective. Assuming that $\ddot{P}\lneq \Delta(P\times P)$ would be a contradiction to the previous statement. This proves that the group-graded Morita equivalence is basic.
\end{proof}
\section{Source algebras of block extensions} \label{sec3}

\begin{nim} Let $G'=N_G(Q)$ and $H'=N_H(Q)$ as in Section \ref{sec4}. Similarly to \ref{n:source}, let $Q_{\delta'}\leq H'_{b'}$, $P_{\gamma'}\leq G'_{b'}$ be defect pointed groups with source idempotents $i'\in\gamma'$, $j'\in\delta '$ such that $j'=i'j'=j'i'$. Set
\[A'_{\gamma'}=i'A'i',  \qquad A'_{\delta'}=j'A'j', \qquad B'_{\delta'}=j'B'j', \qquad B'_{\gamma'}=i'B'i'.\]
It is well known (see \cite[Theorem 5 and Corollary 7]{ALR}) that there is an injective algebra map $B'_{\delta'} \to B_\delta$. We may adapt the argument to obtain a graded version this property, with an additional condition. Note that this condition obviously holds when $\bar G$ is a $p'$-group, or when $G=PH$ and $P$ is abelian, as in \cite{Zhou2}.
\end{nim}

\begin{proposition} \label{p:tibi} With the above notations assume that $N_G(Q_{\delta})\setminus N_G(P_{\gamma})\subseteq H$.  Then there is a unital injective homomorphism  \[A'_{\gamma'}\rightarrow A_{\gamma}, \qquad a'\mapsto fa'\] of $G/H$-graded algebras, where $f$ is a primitive idempotent in $(\Od H)^{G'}$ with defect $P$, which verifies $bb'f=f=fbb'$ and $i'f=fi'=i$.
\end{proposition}

\begin{proof}  Note that the assumption implies that
\[G/H\simeq G'/H'\simeq N_G(Q_{\delta})/N_H(Q_{\delta})\simeq N_G(P_{\gamma})/N_H(P_{\gamma}).\]
Further, let $K_H,$ $K$ and $K_{H'}$ be  the inverse image in $G\times G$ of $\Delta(G/H\times G/H),$ the inverse image in $G\times G'$ of  $\Delta(G/H\times G'/H')$ and the inverse image in $G'\times G'$ of $\Delta(G'/H'\times G'/H')$, respectively. Then the inclusions
\[\Delta(P\times P)\leq N_{K_H}(\Delta(P\times P))\leq K_{H'}\leq K\leq K_H\]
of groups hold.

Now, the indecomposable $\mathcal{O}K_H$-module $\mathcal{O}Hb$, having $\Delta(P\times P)$ as vertex, is the Green correspondent of the indecomposable $\mathcal{O}K_{H'}$-module $\mathcal{O}H'b',$ which also has vertex $\Delta(P\times P).$  Using again the same correspondence, we determine a unique indecomposable $\mathcal{O}K$-module with vertex $\Delta(P\times P),$ say $X,$ that lies in $\Res_{K}^{K_H}(\mathcal{O}Hb)$ and in $\Ind_{K_{H'}}^{K}(\mathcal{O}H'b').$ Explicitly, we have that $X=\mathcal{O}Hf$ for a primitive idempotent $f$ lying in $(\mathcal{O}H)^{G'}$ such that $\Br_P(f)\neq 0.$ Note that $X\mid \Ind_{H\Delta(P\times P)}^{K}(Z)$ for an indecomposable $\mathcal{O}H\Delta(P\times P)$-module $Z$ that has vertex $\Delta(P\times P).$

We have that $\mathcal{O}H'i'$ is an indecomposable $\mathcal{O}H'\Delta(P\times P)$-module with vertex $\Delta(P\times P)$ that is also a source module of $\mathcal{O}H'b',$ where recall that $i'$ is a primitive idempotent of $(\mathcal{O}H')^{P}.$ We consider the $\mathcal{O}H\Delta(P\times P)$-module
\[M:=X\otimes_{\mathcal{O}H'}\mathcal{O}H'i',\] and we claim that $M$ is an indecomposable module with vertex $\Delta(P\times P)$, and that it is also a source module of $\mathcal{O}Hb.$ Indeed, we have that $M$ is a direct summand of  $\mathcal{O}Hi'$, and  the isomorphism
\[\mathcal{O}Hi'\simeq \Ind_{H'\Delta(P\times P)}^{H\Delta(P\times P)}(\mathcal{O}H'i')\]
of $\mathcal{O}H\Delta(P\times P)$-modules  shows that
\[M\mid \Ind_{H'\Delta(P\times P)}^{H\Delta(P\times P)}(\mathcal{O}H'i').\]
Let us emphasise that we have the inclusion
\[N_{H\Delta(P\times P)}(\Delta(P\times P))\leq H'\Delta(P\times P),\] and this forces the Green correspondent of $\mathcal{O}H'i'$ to lie in $M.$ It follows that $M=M'\oplus M''$ where $M'$ is an indecomposable $\mathcal{O}H\Delta(P\times P)$-module with vertex $\Delta(P\times P),$ while $M''$ is a direct sum of indecomposable modules with vertices strictly smaller than $\Delta(P\times P).$

We have, by our hypothesis,
\begin{align*}
M&\mid \Res^K_{H\Delta(P\times P)}(X)\mid \Res^K_{H\Delta(P\times P)}(\Ind_{H\Delta(P\times P)}^{K}(Z))\\
&=\Sum_{(x,x')\in [H\Delta(P\times P)/K \setminus H\Delta(P\times P)]}\Ind_{H\Delta(P\times P)\cap (H\Delta(P\times P))^{(x,x')}}^{H\Delta(P\times P)}(Z^{(x,x')})\\
&=\Sum_{(x,x')\in [H\Delta(P\times P)/K \setminus H\Delta(P\times P)]}Z^{(x,x')}.
\end{align*}
Consequently $M''=0$ and then $M=M'$ is the Green correspondent of $\mathcal{O}H'i',$ this forces $fi'=i'f$ to be  primitive idempotent of $(\mathcal{O}H)^P$ with $bfi'=fi'=fi'b,$ hence $M$   is a source module of $\mathcal{O}Hb.$ We may therefore assume that $fi'=i$ and that $M=\mathcal{O}Hi.$

With these notations we get the isomorphism
\[\mathcal{O}H'i\simeq \mathcal{O}H'f\otimes_{\mathcal{O}H'}\mathcal{O}H'\simeq \mathcal{O}H'i'\]
of $\mathcal{O}H\Delta(P\times P)$-modules, hence $\mathcal{O}H'i'\mid\Res^{H\times 1}_{H'\times 1}(\mathcal{O}Hi).$ We obtain a $P$-algebra homomorphism
\[B'_{\gamma'}\simeq \End_{\mathcal{O}H'}(\mathcal{O}H'i')\to \End_{\mathcal{O}H}(\mathcal{O}Hi)\simeq B_{\gamma},\] which is given by $a'\mapsto a'f=fa'$ for any $a'\in B'_{\gamma'}$. Finally, this map extends in an obvious way to a homomorphism \[A'_{\gamma'}\to A_{\gamma},\] of $P$-interior $\bar G$-graded algebras  which,   since
\[\mathcal{O}G'i'\simeq \mathcal{O}G'\otimes_{\mathcal{O}H'}\mathcal{O}H'i'\mid \mathcal{O}G'\otimes_{\mathcal{O}H'}\mathcal{O}Hi \mid \mathcal{O}Gi,\]
it is also injective.
\end{proof}
\section{Lifting Morita equivalences} \label{sec5}

In this section we give a general technical lemma which is useful to lift Morita equivalences between $1$-components to $G$-graded Morita equivalences. The notation below will be used only in this section.

\begin{nim} Let $G$ be a finite group, and let $A$ and $A'$  be  strongly $G$-graded $\Od$-algebras with $1$-components $B:=A_1$ and $B'=A_1'$. Denote $\bar{A}:=A/J_{\mathrm{gr}}(A)$ and $\bar{A'}:=A'/J_{\mathrm{gr}}(A')$.
\end{nim}

\begin{nim} We consider the diagonal subalgebra
\[\Delta:=\Delta(A\otimes {A'}^{\mathrm{op}}) = \sum_{g\in G}A_g\otimes A'_{g^{-1}}\]
of $A\otimes {A'}^{\mathrm{op}}$, and let
\[\bar{\Delta}:=\Delta/J_{gr}(\Delta)\simeq \Delta(\bar{A}\otimes_k\bar{A'}^{\mathrm{op}}).\]
\end{nim}

\begin{nim}  Let $M$ be a $G$-invariant $\Delta_1$-module (that is, a $G$-invariant $(B,B')$-bimodule). As in \ref{ss:Clifford-ext}, this means that $A_g\otimes_B M\otimes_B A_{g^{-1}}\simeq M$ as $(B,B)$-bimodules. Let $\bar M:=M/J(\Delta_1)M$, and consider the $G$-graded endomorphism algebras
\[\mathcal{D}:=\End_{\Delta}(\Delta\otimes_{\Delta_1}M)^{\mathrm{op}}, \qquad \bar{{\mathcal{D}}}:=\End_{\bar{\Delta}}(\bar{\Delta}\otimes_{\bar{\Delta_1}}\bar{M})^{\mathrm{op}}.\]
As in \ref{ss:Clifford-ext}, the group extension
\begin{equation*}\label{eq1}\xymatrix{1\ar[r] & \mathcal{D}_1^{\times}\ar[r] &\hU(\mathcal{D})\ar[r]&G\ar[r] &1} \end{equation*}
is  the Clifford extension of the $\ \ \Delta_1$-module $ \ \ M$, and the  group extension $\ \ \hU(\mathcal{D} / J_{\mathrm{gr}}(\mathcal{D}))\ \ \ $ is the {residual Clifford extension} of $M$.
\end{nim}

\begin{lemma} \label{lemaDade*} With the above notations, assume that the following conditions hold.
\begin{enumerate}
\item[{\rm(1)}] $M$  induces a Morita equivalence between $B$ and $B'$.
\item[{\rm(2)}] $\bar{M}$ is a simple $\bar{\Delta_1}$-module.
\item[{\rm(3)}] The algebra $\mathcal{D}_1$ is commutative.
\item[{\rm(4)}] $\End_{\bar{A}}(\bar{A}\otimes_{\bar{B}}\bar{M})^{\mathrm{op}}\simeq \bar{A}'$.
\item[{\rm(5)}] For a Sylow $p$-subgroup $P$ of $G$, $M$ extends to a $\Delta_P$-module.
\end{enumerate}
Then $A\otimes_B M$ induces a $G$-graded Morita equivalence between $A$ and $A'$.
\end{lemma}

\begin{proof} Since $\bar{M}$ is a simple $\bar{\Delta_1}$-module,  we have that $\bar{{\mathcal{D}}}\simeq {\mathcal{D}}/J_{gr}({\mathcal{D}})$ (see \cite[Lemma 2.4]{Ma3}), which means that the Clifford extension $\hU(\bar E)$ of $\bar M$ is also the residual Clifford extension of $M$.

Condition (4) implies that $\bar{A}\otimes _{\bar{B}}\bar{M}$ is a $G$-graded $(\bar{A},\bar{A}')$-bimodule, so by \cite[Lemma 1.6.3]{Ma}, $\bar{M}$ extends to a $\bar{\Delta}$-module. By \cite[(1.7)]{D84} it follows that  the Clifford extension $\hU(\bar{\mathcal{D}})$ of $\bar M$ splits. Consequently, by \cite[Theorem 2.8]{D84}, we deduce that for any Sylow $q$-subgroup $Q$ of $G$, where $q\neq p$,  the Clifford extension
\begin{equation*} \label{eq4} \xymatrix{1\ar[r]& \mathcal{D}_1^{\times}\ar[r] &\hU(\mathcal{D}_Q)\ar[r]& Q\ar[r] &1}\end{equation*}
of $M$ splits. By assumption (5) we have that the extension
\begin{equation*} \label{eq4.1} \xymatrix{1\ar[r]& \mathcal{D}_1^{\times}\ar[r] &\hU(\mathcal{D}_P)\ar[r]& P\ar[r] &1}\end{equation*}
also splits.

Since $\mathcal{D}_1$ is commutative, we deduce that the Clifford extension $\hU(\mathcal{D})$ splits (see, for instance \cite[Theorem 7.2]{D84}), hence $M$ extends to a $\Delta$-module. It follows by \cite[Theorem 5.1.2]{Ma} that
\[A\otimes_B M\simeq M\otimes_{B'}A'\simeq (A\otimes (A')^{\mathrm{op}})\otimes_{\Delta}M\]
induces a $G$-graded Morita equivalence between $A$ and $A'$.
\end{proof}

\section{Extensions of nilpotent and inertial blocks}\label{sec7}

\begin{nim}Let $b$ be a $G$-invariant block of $\mathcal{O}H$ which is with defect pointed group $Q_{\delta}$ in $H_{\{b\}}$, as in the Introduction, and let $j\in\delta$. Because of the Fong-Reynolds reduction from Section \ref{sec4}, the notations in this section and in the last one are as follows:
\[A':=\mathcal{O}N_G(Q_\delta)b_\delta, \qquad  B':=\Od N_H(Q_\delta)b_\delta; \qquad C':=\Od QC_H(Q)b_\delta.\]
and we regard $A'$ as an $\tilde E$-graded algebra with $1$-component $C'$.
Let $A'_{\gamma'}=i'\mathcal{O}N_G(Q_\delta)i'$ be the source algebra of the block extension $A'$.
\end{nim}

\subsection{Extensions of nilpotent blocks}\label{ss:sec7} \medskip

We assume in this subsection that $B=\Od Hb$ is a nilpotent block, and we refer to \cite[Chapter 7]{The} for a comprehensive presentation of Puig's theorem on the source algebras of nilpotent blocks.

\begin{nim}  First, note that  $N_H(Q_{\delta})=QC_H(Q)$, hence  $\tilde E=\bar G$ and $C'=B'$. Moreover,  the  source algebra $B_\delta =jBj$ has an $\Od$-simple, $Q$-stable  subalgebra $S_\delta$ such that \[jBj\simeq S_\delta\otimes \Od Q;\qquad S_\delta\simeq \End_\mathcal{O}(V_\delta)\]
where $V_\delta$ is the unique (up to isomorphism) $\mathcal{O}$-simple $jBj$-module, and $p\nmid \mathrm{rank}_{\Od}(S_\delta),$ that is, $S_\delta$ is a Dade $Q$-algebra; denoting $\bar{V}_\delta=k\otimes_{\Od}V_\delta$ and $\bar{S}=k\otimes_{\Od}S$, we have $\bar{S}_\delta\simeq jBj/J(jBj)$ is a simple $k$-algebra, and $\bar{V}_\delta$ the unique simple $\bar{S}_\delta$-module. The Morita equivalence between $jBj$ and $\Od Q$ is given  by the functor
\[V_\delta\otimes_{\Od}-:\Od Q\textrm{-}\mathrm{Mod}\rightarrow jBj\textrm{-}\mathrm{Mod}.\]
Since $B$ is Morita equivalent to $\Od Q$, there is a unique $\Od$-simple $B$-module $U$, and  to $U$ it corresponds a unique $\Od$-simple $jBj$-module $V_\delta$ such that $U=Bj\otimes_{jBj}V_\delta$.
\end{nim}

The following Lemma should be compared to \cite[1.11 and 1.15]{KuPu}

\begin{lemma} \label{l:Cliffordext} The residual Clifford extensions of $U$ and $V_\delta$ are isomorphic to  $k_{\alpha}\tilde E$ as $\tilde E$-graded algebras, where $\alpha\in Z^2(\tilde E,k^{\times})$ is defined in \ref{d:cocycle-beta}.
\end{lemma}

\begin{proof} From the $\tilde E$-graded Morita equivalence between $A$ and $jAj$ we obtain the isomorphism
\[\End_A(A\otimes_B U)^{\mathrm{op}}\simeq\End_{jAj}(jAj\otimes_{jBj}V_\delta)^{\mathrm{op}}\]
of $\tilde E$-graded algebras, see \cite[Corollary 5.1.4]{Ma}. In particular, $U$ and $V_\delta$ have isomorphic residual Clifford extensions.

Observe that $B^Q$ and $N_G(Q_\delta)$ generate a $\bar G$-graded subalgebra of $A$, isomorphic to $B^Q\otimes _{C_H(Q)}N_G(Q_\delta)$, with $1$-component $B^Q\otimes _{C_H(Q)}Q$. We have that $B^Q\otimes _{C_H(Q)}Qj$ is an $N_G(Q_\delta)$-invariant indecomposable projective $B^Q\otimes _{C_H(Q)}Q$-module.  By \cite[Proposition 6.2]{CoMaTo}, its residual Clifford extension is isomorphic to  $k_{\alpha}\tilde E$ as $\tilde E$-graded algebras, since $\tilde E\simeq \bar G$ in our situation. But we have a unital injective map
\[jB^Q\otimes _{C_H(Q)}N_G(Q_\delta)j \to jAj\]
of $\bar G$-graded algebras, sending $b\otimes x $ to $bx$ for all $b\in B^Q$ and $x\in N_G(Q_\delta)$. By taking quotients modulo the graded Jacobson radicals, we obtain a homomorphism
\[k_\alpha \tilde E \to \bar A_\delta:=A_\delta/J_{\mathrm{gr}}(A_\delta)\]
of $\tilde E$-graded algebras. In particular, $\tilde E$ acts on $\bar S_\delta$ and $\bar A_\delta$ is a crossed product of the form $\bar A_\delta\simeq \bar S_\delta\otimes k_\alpha \tilde E$.

To prove the lemma, we need to show the isomorphism
\[\End_{\bar A_\delta}(\bar A_\delta\otimes_{\bar S_\delta}\bar V_\delta)^\mathrm{op}\simeq k_\alpha \tilde E\]
of $\tilde E$-graded algebras. But this is the same to show that the $\bar S_\delta$-module structure of $\bar V_\delta$ extends to a module structure over the diagonal subalgebra
\[\Delta(\bar A_\delta\otimes (k_\alpha \tilde E)^\mathrm{op})\simeq \bar S_\delta\otimes k\tilde E.\]
This condition, in turn, is equivalent to the splitting of the Clifford extension $\End_{\bar S_\delta\otimes k\tilde E}((\bar S_\delta\otimes k\tilde E)\otimes_{S_\delta}\bar V_\delta)$.

Since $\bar V_\delta$ is an $\tilde E$-invariant indecomposable endopermutation $kQ$-module, its residual Clifford extension splits, by Dade's theorem. Therefore, it is enough to show that there is an $\tilde E$-graded algebra homomorphism
\[\End_{\bar S_\delta\otimes k\tilde E}((\bar S_\delta\otimes k\tilde E)\otimes_{S_\delta}\bar V_\delta) \to \End_{kL}(kL\otimes_{kQ}\bar V_\delta).\]
Indeed, such a homomorphism exists, and it is defined (in a way similar to \cite[Proposition 5.4]{CoMaTo}) as follows. Let $\tilde x\in \tilde E$, where $x\in N_G(Q_\delta)$,  and let $y\in L$ such that $\pi(y)=\tilde x$. An element of degree $\tilde x$ from the domain is a $k$-linear map $f:V_\delta\to V_\delta$ satisfying $f\circ s=s^{\tilde x}\circ f$ for all $s\in \bar S_\delta$, while an element of degree $\bar y\in L/Q$ from the codomain is a $k$-linear map $f':V_\delta\to V_\delta$ such that $f'(uv)=u^yf'(v)$ for all $v\in V$ and $u\in Q$. It is straightforward to check, by using Remark \ref{r:conjQ}, that the restriction of scalars $f\mapsto f$, via $kQ\to S$, gives the required $\tilde E$-graded algebra homomorphism.
\end{proof}

\begin{nim} We know  by \cite[Proposition 6.5]{KuPu}, or by \cite[Theorem 2]{Ca} that $BP=b\Od PH$ is also a nilpotent block of $\Od PH$. Then, by Puig's results on nilpotent blocks, we have:
\begin{itemize}
\item[$\bullet$] the $\bar P\simeq P/Q$-graded source algebra $iBPi$ has an $\Od$-simple, $P$-stable  subalgebra $S_\gamma$ such that \[iBPi\simeq S_\gamma\otimes \Od P;\qquad S_\gamma\simeq \End_\mathcal{O}(V_\gamma)\]
where $V_\gamma$ is the unique $\mathcal{O}$-simple $iBPi$-module, and $p\nmid \mathrm{rank}_{\Od}(S_\gamma),$ so $S_\gamma$ is a Dade $P$-algebra;
\item[$\bullet$] denoting $\bar{V}_\gamma=k\otimes_{\Od}V_\gamma$ and $\bar{S}_\gamma=k\otimes_{\Od}S$, we have that $\bar{S}_\gamma\simeq iBPi/J(iBPi)$ is a simple $k$-algebra, and  $\bar{V}_\gamma$ is the unique simple $\bar{S}_\gamma$-module.
\end{itemize}
\end{nim}

Now, taking into account \ref{subsec82}, Theorem \ref{thmA} b) is a consequence of the following more precise statement.

\begin{thm} \label{t:extnilp} Assume that $B=\Od Hb$ is a nilpotent block. There is a $\bar G$-graded Morita equivalence between $A_\gamma$ and $\Od_{\alpha}L$ induced by $V_\gamma$, or equivalently, an isomorphism
\[A_\gamma \simeq S_\gamma\otimes \Od_{\alpha}L.\]
\end{thm}

\begin{proof} The $\bar P$-graded Morita equivalence between $A_\gamma$ and $\Od P$ restricts to the Morita equivalence $V_\gamma\otimes_\Od -:\Od Q\textrm{-Mod}\to B_\gamma\textrm{-Mod}$.  We aim to use Lemma \ref{lemaDade*} to lift this to a  Morita equivalence
\[V_\gamma\otimes_\Od -:\Od_{\alpha}L \textrm{-Mod}\to A_\gamma\textrm{-Mod}.\]
We have that $\bar A_\gamma:=A_\gamma/J_\mathrm{gr}(A_\gamma)$ is a $\bar G$-graded crossed product with $1$-component $\bar B_\gamma=B_\gamma/J(B_\gamma)\simeq \bar S_\gamma$. The algebra $R:=\Od _\alpha L$ is $\bar G\simeq L/Q$-graded with $1$-component $\Od Q$, and $\bar R:=R/J_{\rm gr}(R)\simeq k_\alpha \bar G$. Let
\[\Delta:=\Delta (A_\gamma\otimes _{\Od}R^{\mathrm{op}}), \qquad  \Delta_1=B_\gamma\otimes _{\Od}(\Od Q)^{\mathrm{op}}, \qquad  \bar{\Delta}=\Delta/J_{gr}(\Delta).\]
We have $\Delta_1\simeq (S_\gamma\otimes \Od Q)\otimes (\Od Q)^\mathrm{op}$, $\Delta_1/J(\Delta_1)\simeq \bar S_\gamma$. From $V_\gamma$ we get the $(iBi,\Od Q)$-bimodule
\[M=M_\gamma:=V_\gamma\otimes \Od Q\]
inducing a Morita equivalence between $iBi$ and $\Od Q$, hence $M$ satisfies condition (1) of Lemma \ref{lemaDade*}. Let $\bar M=M/J(\Delta_1)M$; then $\bar M\simeq \bar V_\gamma$ is a simple $\bar S_\gamma$-module, hence $M$ satisfies condition (2) of Lemma \ref{lemaDade*}. Moreover,
\[\End_{\Delta_1}(M)\simeq \End_{S_\gamma\otimes\Od Q}(V_\gamma \otimes\Od Q)\simeq Z(\Od Q)\]
is commutative, therefore condition (3) of Lemma \ref{lemaDade*} also holds. For the Clifford extensions, observe that
\[\End_{\bar A_\gamma}(\bar A_\gamma\otimes_{\bar B_\gamma}\bar M_\gamma)^{\textrm{op}} \simeq \End_{\bar A_\delta}(\bar A_\delta\otimes_{\bar B_\delta}\bar M_\delta)^{\textrm{op}} \simeq k_\alpha\bar G, \]
as $\bar G$-graded $k$-algebras, where the first isomorphism is a consequence of the $\bar G$-graded Morita equivalence between $A_\delta$ and $A_\gamma$ (see \cite[Propositon 3.2]{Ma2}), while the second isomorphism follows from Lemma \ref{l:Cliffordext}. This gives condition (4) of Lemma \ref{lemaDade*}. Finally, $M=V_\gamma\otimes \Od Q$ has a diagonal action of $P$, because $\Od Q$ is a $P$-algebra and $V_\gamma$ is an $\Od P$-module, so condition (5) of Lemma \ref{lemaDade*} also holds.

It follows that $M$ extends to $\Delta$, and consequently, $A_\gamma\otimes_{B_\gamma}M$ induces an $\bar G$-graded Morita equivalence between $A_\gamma$ and $\Od_{\alpha}L$.
\end{proof}

\subsection{Extensions of  blocks with normal defect groups}\label{ssec4} \medskip

\begin{nim} Consider the block extension $A'=\Od N_G(Q_{\delta})b_{\delta}$, regarded as an $\tilde E$-graded algebra with identity component $C':=\Od QC_H(Q)b_{\delta}$. The block  $b_{\delta}$ has defect pointed group $Q_{\delta'}$ in $N_H(Q_{\delta})$ and, by \ref{propPdefectbdelta} applied for $R=Q$, it has defect pointed group $P_{\gamma'}$ in $N_G(Q_{\delta})$.

We know that the source algebra of  $C'$ is  $j'\Od QC_H(Q)j'\simeq \Od Q$, where $j'\in \delta'$, and the bimodule which gives the Morita  equivalence is $\Od QC_H(Q)j'$, so here we have a particular case of extensions of nilpotent blocks.  By  K\"ulshammer \cite[Theorem A]{Ku}, (see also \cite[Theorem 13]{ALR}), we have that the source algebra of  $B'=\Od N_H(Q_{\delta})b_{\delta}$ is  $B'_{\delta'}=j'B'j'\simeq \Od_\alpha( Q\rtimes E_H(Q_\delta))$,

Let $A'_{\gamma'}=i'A'i'$ be the source algebra of $A'$,  where $i'\in (\Od C_H(Q))^P$. As above, $PC'=\mathcal{O}PC_H(Q)b_\delta$ is a nilpotent block. Let $V_{\gamma'}$ be the unique, up to isomorphism, simple $i'\mathcal{O}PC_H(Q)i'$-module, and let $S_{\gamma'}=\End_\Od(V_{\gamma'})$, which  is a Dade $P$-algebra.  The source algebra of $\Od PC_H(Q)b_{\delta}$ is isomorphic to $S_{\gamma'}\otimes \Od P$, and by Theorem \ref{t:extnilp} and \ref{n:Lpprime} we get:
\end{nim}

\begin{corollary} \label{c:cornormaldef} There is an $\tilde E$-graded Morita equivalence between  $A'_{\gamma'}$ and $\mathcal{O}_\alpha L$ induced by $V_{\gamma'}$. More precisely, there is an isomorphism
\[A'_{\gamma'}\simeq S_{\gamma'} \otimes \mathcal{O}_\alpha L\]
of $P$-interior $\tilde E$-graded algebras, where $L$ is defined in Section \ref{sec2}.
\end{corollary}

\subsection{$p'$-extensions of inertial blocks} \label{sec6}  \medskip

In this subsection we assume that the block $B$ is inertial, and we prove the statements a) and c) of Theorem \ref{thmA}. By Puig \cite[2.16]{Pu2}, the assumption means that
\[B_\delta\simeq S\otimes B'_{\delta'},\]
where $S$ is a Dade $Q$-algebra, unique up to isomorphism. Statement 2) of the next theorem is due to Y.~Zhou \cite[Proposition 3.3]{Zhou1}. Here we give an alternative proof based on Lemma \ref{lemaDade*} (see also the Remark following \cite[Theorem]{Zhou1}).

\begin{thm} \label{l:invCgamma} Assume that the block $B$ is inertial.

{\rm 1)} The bimodule inducing the Morita equivalence between $B$ and $\mathcal{O}_\alpha(Q\rtimes \tilde E_H(Q_\delta))$ is $\bar G$-invariant.

{\rm 2)} If $p$ does not divide the order of $\bar G$,  then $A_\delta$ is an $\tilde E$-graded algebra with identity component $C_\delta$ Morita equivalent to $\mathcal{O}Q$, and there is an isomorphism
\[A_\delta\simeq S\otimes A'_{\delta'}\]
of $\tilde E$-graded algebras.

{\rm 3)} The Clifford extension of the unique simple $C_\delta$-module is isomorphic to the Clifford extension of $\bar V$.
\end{thm}

\begin{proof} 1) Let  $x\in N_G(Q_\delta)$. Then there is $b_x\in (B^Q)^\times$ such that $xjx^{-1}=b_xjb_x^{-1}$, hence $b_x^{-1}x$ commutes with $j$, and let $a_x:=(b_x^{-1}x)j\in N_{\mathrm{hU}(A_\delta)}(Q)$. In particular, $A_\delta$ is a $\bar G$-graded crossed product with $1$-component $B_\delta$.

Since $b_\delta$ is a block of $\mathcal{O}QC_H(Q)$ with defect group $Q$, we have that $A'$ is $\tilde E$-graded, so a similar argument gives homogeneous units $a'_x\in N_{\mathrm{hU}(A'_{\delta'})}(Q)$ such that $A'_{\delta'}$ is an $\tilde E$-graded crossed product, while $B'_{\delta'}$ is an $E_H(Q_\delta)$-graded crossed product, both with $1$-component $C'_{\delta'}:=j'\mathcal{O}QC_H(Q)j'\simeq \mathcal{O}Q$.

Recall from \cite[Proposition 4.2]{CoMaTo} that the correspondences $x\mapsto \overline{a_x}$ and $x\mapsto \overline{a'_x}$  induce isomorphisms
\begin{align}
N_G(Q_\delta)/C_H(Q) & \simeq  N_{\mathrm{hU}(jAj)}(Qj)/C_{(jBj)^\times}(Qj) \label{f:isofus1} \\
    &\simeq N_{\mathrm{hU}(j'A'j')}(Qj)/C_{(j'B'j')^\times}(Qj')  \label{f:isofus2}
\end{align}
between the group of $(G,\bar G)$-fusions of $Q_\delta$, the group of $(A,\bar G)$-fusions of $Q_\delta$, and the group of $(A',\tilde E)$-fusions of $Q_\delta$.

By our assumption, we have a Morita equivalence between $B_\delta$ and $B'_{\delta'}$ given by the functor
\[W \otimes_{\mathcal{O}}-:B'_{\delta'}\textrm{-mod}\to B_\delta\textrm{-mod},\]
where $W$ is an indecomposable endopermutation $\mathcal{O}Q$-module such that $S=\mathrm{End}_{\mathcal{O}}(W)$. Equivalently, the Morita equivalence is induced by the $E_H(Q_\delta)$-graded $(B_\delta,B'_{\delta'})$-bi\-module $W\otimes B'_{\delta'}$. Since $S$ can be chosen to be $N_G(Q_\delta)$-stable (see \cite[3.6]{Zhou1}), we may assume that $W$ is an $N_G(Q_\delta)$-invariant $\mathcal{O}Q$-module. Moreover, $W$ is the unique, up to isomorphism,  $\mathcal{O}$-simple $S$-module.

Let $\phi_x$ be the automorphism of $Q$ given by $x$-conjugation, and note that $a_x(uj)a_x^{-1}=x(uj)x^{-1}$ for all $u\in Q$. Moreover, by \cite[3.6]{Zhou1}, there is $s_x\in S^\times$ such that $s_xus_x^{-1}=\phi_x(u)\cdot 1_S$ for all $u\in Q$.  Conjugation with $a_x$ induces an isomorphism $B_\delta\simeq \mathrm{Res}_{\phi_x}(B_\delta)$ of $Q$-interior algebras, so $a_xs_x^{-1}$-conjugation induces a $Q$-interior algebra automorphism of $B_\delta$. We deduce that the $(B_\delta, B'_{\delta'})$-bimodules $W\otimes B'_{\delta'}$ and $a_x\otimes(W\otimes B'_{\delta'})\otimes {a'_x}^{-1}$ are isomorphic, since $a'_x\otimes B'_{\delta'})\otimes {a'_x}^{-1}\simeq B'_{\delta'}$ as $(B'_{\delta'},B'_{\delta'})$-bimodules, and $a_x\otimes W\otimes {a'_x}^{-1}\simeq W$ as $\mathcal{O}Q$-modules.

2) Observe that $W\otimes B'_{\delta'}$ is an $E_H(Q_\delta)$-graded right $B'_{\delta'}$-module, hence the isomorphism $B_\delta\simeq \mathrm{End}_{(j'B'j')^{\mathrm{op}}}(W\otimes B'_{\delta'})$ defines an $E_H(Q_\delta)$-grading on $B_\delta$, with $1$-component denoted by $C_\delta\simeq S\otimes \mathcal{O}Q$. We may regard $S$ as a subalgebra of $C_\delta$, and $W$ is, up to isomorphism,  the unique $\mathcal{O}$-simple $C_\delta$-module. To show that the $\bar G$-grading of $A_\delta$ can be refined to an $\tilde E$-grading, we verify the conditions \ref{ss:grade-refine} (1), (2) and (3) for $A_\delta$, $B_\delta$ and $C_\delta$.

By Proposition  \label{p:tibi} there is a unital injective homomorphism $A'_{\delta'}\to A_\delta$ which is compatible with the isomorphisms  (\ref{f:isofus1})  and (\ref{f:isofus2}) (see \cite[Proposition 7.4]{Pu2m})   This implies that the condition \ref{ss:grade-refine} (1) holds. Note also that since $A'_{\delta'}\simeq\mathcal{O}_\alpha(Q\rtimes \tilde E)$, a different choice of $a'_x$ means multiplication by an element from  $\mathcal{O}^\times$.

The $\mathcal{O}Q$-module $s_xa_x^{-1}\otimes W$ is isomorphic to $W$. Therefore, the right  $B'_{\delta'}$-module $s_xa_x^{-1}\otimes W\otimes B'_{\delta'}$ is actually $E_H(Q_\delta)$-graded, and still induces a Morita equivalence between $B_\delta$ and $B'_{\delta'}$, such that $\mathrm{End}_{{(j'B'j')^{\mathrm{op}}}}(s_xa_x^{-1}\otimes W\otimes B'_{\delta'})$ identifies with $a_xs_x^{-1}B_\delta s_xa_x^{-1}$, with the same $1$-component as $B_\delta$. In particular, we deduce that $a_xC_\delta a_x^{-1}=C_\delta$, so condition \ref{ss:grade-refine} (2) holds.

Finally,  condition \ref{ss:grade-refine} (3) holds as well, because $a'_xa'_y{a'}_{xy}^{-1}\in {\mathcal O}^\times$ for all $x,y\in N_G(Q_\delta)$.

We now consider the $\tilde E$-graded algebras $A_\delta$ with $1$-component $C_\delta$, $A'_{\delta'}$ with $1$-component $C'_{\delta'}=\Od Q$, and the diagonal subalgebra $\Delta:=\Delta(A_\delta\otimes {A'}_{\delta'}^{\textrm{op}})$ with $1$-component $\Delta_1=C_\delta\otimes (\Od Q)^{\textrm{op}}$. Let $M=W\otimes \mathcal{O}Q$. By the same argument as in 1) above, $M$ is an $\tilde E$-invariant $(C_\delta,\mathcal{O}Q)$-bimodule (that is, $\Delta_x\otimes_{\Delta_1}M\simeq M$ as $\Delta_1$-modules for all $x\in \tilde E$), because $\Od Q$ is $\tilde E$-invariant and $W$ is an $\tilde E$-invariant $\Od Q$-module.

We are going to verify the conditions of Lemma \ref{lemaDade*}. By assumption, $M$ induces a Morita equivalence between $C_\delta$ and $\Od Q$, so condition (1) holds, while condition (5) is trivially true, since $\tilde E$ is a $p'$-group. For condition (2), note that $C_\delta/J(C_\delta)\simeq S$, $\bar\Delta_1=\Delta_1/J(\Delta_1)\simeq \bar S=k\otimes_{\Od}S$ and $\bar M=M/J(\Delta_1)M\simeq\bar W$, so $\bar M$ is a simple $\bar\Delta_1$-module. Since $C_\delta \simeq S\otimes \mathcal{O}Q$, we have that
\[\mathcal{D}_1=\End_{\Delta_1}(M)^\mathrm{op}=\End_{C_\delta\otimes (\mathcal{O}Q)^\mathrm{op}}(W\otimes \mathcal{O}Q)^\mathrm{op} \simeq Z(\mathcal{O}Q)\]
is commutative, so condition (3) also holds.

Finally, for condition (4), observe that $\bar A_\delta=A_\delta/J_{\textrm{gr}} (A_\delta)$ is a crossed product of $\bar S$ and $\tilde E$, while $\bar {A}'_{\delta'}=A'_{\delta'}/J_{\textrm{gr}} (A'_{\delta'})\simeq k_\alpha \tilde E$, where $\alpha\in Z^2(\tilde E,k^\times)$ is defined in \ref{d:cocycle-beta}. By Proposition \ref{p:tibi} there are injective maps  $A'_{\delta'}\to A_\delta$  and $\bar {A}'_{\delta'}\to \bar A_\delta$ of $\tilde E$-graded algebras, hence $\tilde E$ acts on $\bar S$, and $\bar A_\delta$ is a crossed product of the form
\[\bar A_\delta \simeq \bar S\otimes k_\alpha \tilde E.\] To prove that
\[\End_{\bar A_\delta}(\bar A_\delta\otimes _{\bar S}\bar W)^{\mathrm{op}}\simeq k_\alpha \tilde E,\] it is enough to show that the simple $\bar S$-module $\bar W$ extends to the diagonal subalgebra $\Delta(\bar A_\delta\otimes (\bar {A}'_{\delta'})^\mathrm{op})$. Observe that
\[\Delta(\bar A_\delta\otimes (\bar {A}'_{\delta'})^\mathrm{op})\simeq  \Delta((\bar S\otimes k_\alpha \tilde E) \otimes (k_\alpha \tilde E)^\mathrm{op})\simeq \bar S\otimes k\tilde E,\]
hence it is enough to show that $\bar W$ has a structure of a $k\tilde E$-module. We now use the fact that $W$ is an $\tilde E$-invariant indecomposable endopermutation $\mathcal{O}Q$-module, where we regard $Q$ as a normal subgroup of $Q\rtimes \tilde E$. By Dade's theorem \cite[(12)]{D80}, the residual Clifford extension of $W$ splits, and since $\tilde E$ is a $p'$-group, the Clifford extension of $W$ also splits, hence the $\mathcal{O}Q$-module structure of $W$ extends to an $\mathcal{O}(Q\rtimes \tilde E)$-module structure (see \cite[Theorem 6.7]{D84}). In particular, $W$ is an $\mathcal{O}\tilde E$-module and hence $\bar W$ is an $k\tilde E$-module.

3) Under the $\tilde E$-graded Morita between $A_\delta$ and $A'$, the unique simple $C_\delta$-module $\bar W$ corresponds to the unique simple $C'$-module $\bar V$. The statement follows from the fact that group graded Morita equivalences preserve Clifford extensions (see \cite[Theorem 5.1.18]{Ma}).
\end{proof}



\section{Extended local categories}  \label{s:fusions}

In this section we introduce fusions in the general context of $\bar{G}$-graded $G$-interior algebras, and study their properties. We generalize here some notions and results from \cite{CoMaTo}, where only automorphisms of $P$-groups are considered.

\begin{nim} \label{subsec81}  Recall that there is an extended local category denoted $\mathcal{E}_{(b,H,G)}$ introduced by Puig and Zhou in \cite{PuZhIII}. The objects are the local pointed subgroups included in $P_{\gamma}$ and the set of morphisms $\Hom_{\mathcal{E}_{(b,H,G)}}(R_{\epsilon}$, $T_{\rho})$ from  $R_{\epsilon}$ to $T_{\rho}$ (here $T,R\leq P$ and $\rho\subseteq B^T, \epsilon\subseteq B^R$ are local points) is formed by  the pairs
$(c_x,\bar{x})$ with $\bar{x}\in\bar{G}$, such that
\[c_x:R\rightarrow T\quad c_x(u)={}^xu\]
for any $u\in R$; here  ${}^x R_{\epsilon} \leq T_{\rho}$, with $\epsilon$ and $\rho$ local points such that $T_{\rho}\leq P_{\gamma}$ and $R_{\epsilon}\leq P_{\gamma}$. If $R=T$ the automorphism group is
\[\Aut_{\mathcal{E}_{(b,H,G)}}(R_{\epsilon}):=E_G^{\bar{G}}(R_{\epsilon})\simeq N_G(R_{\epsilon})/C_H(R).\]
For brevity, we denote this category by $\mathcal{E}$, and if $P'_{\gamma'}$ is some defect pointed group in $G'_{\{b'\}}$ on $B'$ (which is a  $G'$-algebra), we have a similar extended local category $\mathcal{E}'$.
\end{nim}

\begin{nim} \label{subsec83} We introduce $(A,\bar{G})$-fusions and $(G,\bar{G})$-fusions in a more general context.  Let $\bar{G}=G/H$, let $A$ be a $\bar{G}$-graded $G$-interior algebra such that $B:=A_1$ is an $H$-interior $G$-algebra. Let $P$ be a $p$-subgroup of $G$, and let $P_{\gamma}$ be a local pointed group on the  $G$-algebra $B$ with $i\in\gamma$. We will assume that the structural map $P\rightarrow A^{\times}$ is injective. In particular, $Ai$ becomes a $\bar{G}$-graded $(A,\Od P)$-bimodule, with $(\End_A(Ai))^{\mathrm{op}}\simeq iAi$ as $\bar{G}$-graded $P$-interior algebras.

For two pointed subgroups  $R_{\epsilon}\leq P_{\gamma}$ and $T_{\rho}\leq P_{\gamma}$ we denote
\begin{align*}
N_P(R,T)  &=   \{u\in P\mid {}^uR\leq T\}   \\
N_G(R_{\epsilon},T_{\rho})  &=  \{g\in G \mid  {\ }^gR_{\epsilon}\leq T_{\rho}\} \\
&= \{g\in N_G(R,T) \mid \text{ for } l\in\rho \text{ there is } {\ }^gj\in{\ }^g\epsilon~~  \\  &\ \hspace{45mm} \text{ such that }  {\ }^gj=l\cdot {}^g j={}^gj\cdot l\}.
\end{align*}
For $i\in\gamma$ choose $j\in\epsilon$ and $l\in\rho$ such that
\[j=i j=j i, \qquad l=i l=l i.\]
In particular, to $R_{\epsilon}$ corresponds the $\bar{G}$-graded $(A,\Od R)$-bimodule $Aj$, while to $T_{\rho}$ corresponds the $\bar{G}$-graded $(A,\Od T)$-bimodule $Al$. Denote by $\Inj(R,T)$ the set of injective group homomorphisms from $R$ to $T$.
\end{nim}

\begin{definition} a) The set of $\bar{G}$-{\it fusions} from $R$ to $T$ is
\[\Hom^{\bar{G}}(R,T)=\{(\varphi,\bar{g})\mid \varphi\in \Inj(R,T), \ \bar{g}\in\bar{G}, \overline{\varphi(u)}={}^{\bar{g}}\bar{u}, \ \forall u\in R\}.\]

b) The set of {\it interior} $\bar{G}$-{\it fusions} from $R$ to $T$ is
\[\Int^{\bar{G}}(R,T)=\{(c_v,\bar{g})\mid v\in N_P(R,T),\ \bar{g}\in\bar{G}, \ \overline{c_v(u)}={}^{\bar{g}}\bar{u}, \ \forall u\in R\}\]
\end{definition}

\begin{remark} a) We have the inclusions
\[\Int^{\bar{G}}(R,T)\subseteq \Hom^{\bar{G}}(R,T)\subseteq \Inj(R,T)\times \bar{G}.\]
If $(\varphi,\bar{g})\in \Hom^{\bar{G}}(R,T)$, then $\bar{g}\in N_{\bar{G}}(\bar{R},\bar{T})$.

b) If $R=T=P$, then $\Hom^{\bar{G}}(P,P)=\Aut^{\bar{G}}(P)$, and $\Int^{\bar{G}}(P,P)=\Int^{\bar{G}}(P)$, see \cite[Definition 3.2]{CoMaTo}.
\end{remark}

\begin{definition} With the above notations, assume $R\simeq T$. We define the set
\[N_{\hU(A)}(Rj, Tl):=\{a\in A^{\times}\cap A_{\bar{g}}\mid \bar{g}\in \bar{G},\  a(Rj)a^{-1}=Tl\}.\]
If $Rj=Tl=Pi$, then $N_{\hU(A)}(Pi)$ is a group, defined in \cite{CoMaTo}.
\end{definition}

\begin{nim} \label{ss:shift} Before giving the next definition, we recall the change of gradings construction introduced in \cite[2.2]{CoMaTo} (and used intensively in \cite{CoMaTo}) in the particular case of the $\bar G$-graded $(A,\mathcal{O}T)$ bimodule $Al$. Note that here we regard $\mathcal{O}T$ as a $\bar G$-graded subalgebra of $\mathcal{O}G$, with components  $(\mathcal{O}T)_{\bar x}=\mathcal{O}(\bar x\cap T)$ for $\bar x\in TH/H$, and $(\mathcal{O}T)_{\bar x}=0$ for $\bar x\notin TH/H$.

For $\bar g\in \bar G$, $(\mathcal{O}T)^{\bar g^{-1}}$ is the $\bar G$-graded algebra with components $((\mathcal{O}T)^{\bar g^{-1}}_{\bar x} = (\mathcal{O}T)_{\bar g\bar x\bar g^{-1}}$ for all $\bar x\in \bar G$. Then, $(Al)(\bar g^{-1})$ is the $\bar G$-graded $A, (\mathcal{O}T)^{\bar g^{-1}})$-bimodule with components $(Al)(\bar g^{-1})_{\bar x}=(Al)_{\bar x\bar g^{-1}}$ for all $\bar x\in \bar G$.

Now, if $(\varphi, \bar g)\in \Hom^{\bar G}(R,T)$, then $\varphi$ induces a homomorphism $\mathcal{O}R\to (\mathcal{O} T)^{\bar g^{-1}}$ of $\bar G$-graded algebras, hence, by restriction of scalars, $(Al)(\bar g^{-1})$ becomes a $\bar G$-graded $(A,\mathcal{O}R)$-bimodule, which is denoted by $(Al)(\bar g^{-1})_\varphi$.
\end{nim}

\begin{definition}\label{defn1}
a) The set of $(A,\bar{G})$-fusions from $R_{\epsilon}$ to $T_{\rho}$ is
\begin{align*}
F_A^{\bar{G}}(R_{\epsilon}, T_{\rho})&=\{(\varphi, \bar{g})\in\Hom^{\bar{G}}(R,T)\mid Aj \text{ is a summand of } (Al)(\bar{g}^{-1})_{\varphi} \\
    &\ \hspace{20mm} \text{ as } \bar{G}\text{-graded } (A,\Od R)\text{-bimodules}\}.
\end{align*}

b) The set of $(G,\bar{G})$-fusions  from $R_{\epsilon}$ to $T_{\rho}$ is
\[E_G^{\bar{G}}(R_{\epsilon}, T_{\rho})=\{(c_x,\bar{x})\mid x\in N_G(R_{\epsilon}, T_{\rho}),\ \overline{c_x(u)}={}^{\bar{x}}\bar{u},\ \forall u\in R\}.\]
Note that in the case of a block extension $A=\Od Gb$, this set is actually $\Hom_{\mathcal{E}}(R_{\epsilon}, T_{\rho})$.
\end{definition}

\begin{remark} If $R_{\epsilon}=T_{\rho}=P_{\gamma}$ then $F_A^{\bar{G}}(P_{\gamma}, P_{\gamma})=F_A^{\bar{G}}(P_{\gamma})$ and
\[E_G^{\bar{G}}(P_{\gamma},P_{\gamma})=E_G^{\bar{G}}(P_{\gamma})\simeq N_G(P_{\gamma})/C_H(P),\] see
\cite[Definition 3.3]{CoMaTo}.
\end{remark}

\begin{lemma}\label{lem45} With the above assumptions, if $R\simeq T$, then there is a surjective map
\[\Phi:N_{\hU(A)}(Rj, Tl)\rightarrow F_A^{\bar{G}}(R_{\epsilon}, T_{\rho}),\]
which induces a bijective map
\[\overline{\Phi}:N_{\hU(A)}(Rj, Tl)/\operatorname{ker} \Phi\rightarrow F_A^{\bar{G}}(R_{\epsilon}, T_{\rho}).\]
\end{lemma}

\begin{proof} Note that here $\operatorname{ker}\Phi$ is the equivalence relation induced by $\Phi$, so
\[N_{\hU(A)}(Rj, Tl)/ \operatorname{ker}\Phi:=\]\[\left\{\{a'\in N_{\hU(A)}(Rj, Tl)\mid \Phi(a')=\Phi(a)\}\mid a\in N_{\hU(A)}(Rj, Tl)\right\}.\]
The proof is an adaptation of \cite[Proposition 3.5]{CoMaTo}. Let $\Phi$ be defined by
\[\Phi(a)=(\varphi_a, \bar{g})\]
for any $a\in A^{\times}\cap A_{\bar{g}}$ such that $ a(Rj)a^{-1}=Tl$, where $\varphi_a:R\rightarrow T$ is a bijection  given by the conjugation with $a$. 
To see why we can do this,  note that we have assumed that the restriction $P\to A^\times$ of the structural map of $A$ is injective. Since $R\le P$ and $j\in A^R$, we may regard $jAj$ as an $R$-interior algebra, so we have a group homomorphism $R\to (jAj)^\times$.  The image of this last homomorphism is $Rj=jR=jRj$, and consequently, we obtain a surjective group homomorphism $R\to Rj$. Since $Aj$ is a projective right  $\mathcal{O}R$-module, it is also a free $\mathcal{O}R$-module, which implies that the homomorphism $R\to Rj$ is also injective. A similar argument holds for $Tl$, and by assuming that $T\simeq R$, we get  $R\simeq Rj\simeq Tl\simeq T$ as subgroups in $A^\times$. Then it is also clear that $Aj$ is isomorphic to $(Al)(\bar{g}^{-1})_{\varphi_a}$ as  $\bar{G}$-graded $(A,\Od R)$-bimodules, hence $\Phi$ is a well-defined map.

Next we verify  that $\Phi$ is surjective. For this, let $(\varphi,\bar{g})\in F_A^{\bar{G}}(R_{\epsilon}, T_{\rho})$ such that there is an isomorphism $f:Aj\rightarrow (Al)(\bar{g}^{-1})_{\varphi}$  of $\bar{G}$-graded $(A,\Od R)$-bimodules.  Since $Aj$ and $(Al)(\bar{g}^{-1})_{\varphi}$ are  direct summands of $A$ as left $A$-modules,  it follows by the Krull-Schmidt theorem that there is $\overline{f}\in\Aut_A(A)$ such that $\overline{f}|_{Aj}=f$. Then there is $a\in A^{\times}$ such that $\overline{f}(b)=ba$ for any $b\in A$. In particular,
\[f:Aj\rightarrow (Aj)(\bar{g}^{-1})_{\varphi}\quad f(bj)=bja\]
for any $b\in A$. Since $f$ is a homomorphism of $\bar{G}$-graded $(A,\Od R)$-bimodules, we obtain that $a$ is an homogeneous unit; that is,  $a\in A^{\times}\cap A_{\bar{g}}$ for some $\bar g\in \bar G$. We obtain that
$aRja^{-1}=Tl$ and $\Phi(a)=(\varphi_a,\bar{g})=(\varphi, \bar{g}).$
\end{proof}

We return to the case of the block extension $A=b\Od G$.

\begin{proposition} \label{prop44} 
Let $T_{\rho}\leq P_{\gamma}$ and $R_{\epsilon}\leq P_{\gamma}$ be  local pointed groups on $B=\Od Hb$ such that $R\simeq T$. Then there is a bijection
\[E_G^{\bar{G}}(R_{\epsilon}, T_{\rho})\rightarrow F_A^{\bar{G}}(R_{\epsilon}, T_{\rho}).\]
\end{proposition}

\begin{proof} We will show the existence of two  bijections
\[E_G^{\bar{G}}(R_{\epsilon}, T_{\rho})\rightarrow N_{\hU(A)}(Rj,Tl)/\operatorname{ker} \Phi\rightarrow F_A^{\bar{G}}(R_{\epsilon}, T_{\rho}).\]
The second bijection exists since our algebra satisfies the hypotheses of Lemma \ref{lem45}.

For the first bijection let $(\varphi_g,\bar{g})\in E_G^{\bar{G}}(R_{\epsilon}, T_{\rho})$. Since $R\simeq T$, we obtain that ${}^gR_{\epsilon}=T_{\rho}$, hence ${}^g\epsilon=\rho$ and ${}^gR=T$. It follows that for $j\in \epsilon$ and $l\in \delta$, there is $b_1\in (B^T)^{\times}$ such that $gjg^{-1}=b_1lb_1^{-1}$, hence $l=b_1^{-1}gjg^{-1}b_1$. We denote
\begin{equation}\label{eq*}
a=b_1^{-1}g
\end{equation}
which is a homogeneous invertible element. Then
\[aRja^{-1}=b_1^{-1}gRjg^{-1}b_1=b_1^{-1}({}^gR)({}^gj)b_1=b_1^{-1}T( {}^gj)b_1=Tl.\]
Let
\[\Theta: E_G^{\bar{G}}(R_{\epsilon}, T_{\rho})\rightarrow N_{\hU(A)}(Rj,Tl)/ \operatorname{ker}  \Phi, \quad \Theta(\varphi_g,\bar{g})=[a]_{\operatorname{ker}  \Phi},\]
where $a$ is obtained in (\ref{eq*}) and $\Phi$ is the map from Lemma \ref{lem45}; here $[a]_{\operatorname{ker} \Phi}$ is the set
\[\{a'\in N_{\hU(A)}(Rj,Tl)\mid \Phi(a)=\Phi(a')\}.\]
Our definition of $\Theta$ does not depend on the choice of $b_1$, and the injectivity of $\Theta$ follows by straightforward verification.

To show that $\Theta$ is surjective, let $\bar{a}\in N_{\hU(A)}(Rj, Tl)/\Ker \Phi $.
By Lemma \ref{lem45} we have $(\varphi_{a_g},\bar{g})\in F_{A}^{\bar{G}}(R_{\epsilon}, T_{\rho})$, and denote by $\varphi$ the map $\varphi_{a_g}$. By Definition \ref{defn1}, since $\varphi$ is an isomorphism, we obtain
\[Aj\simeq (Al)(\bar{g}^{-1})_{\varphi}\] as $\bar{G}$-graded $(A,\Od R)$-bimodules. Next we  mimic  the proof of  \cite[7.2, 7.3]{L} in our graded context. Note that
\[jAj\simeq (jAl)(\bar{g}^{-1})_{\varphi}\]
as $\bar{G}$-graded $(\Od R,\Od R)$-bimodules. Since $\Br_R(j)\neq 0$ and $A=\Od Gb$, it follows that $jAj$ has a direct summand isomorphic to $\Od R$ as $(\Od R,\Od R)$-bimodules, hence $jAl$ has a direct summand isomorphic to $(\Od R)(\bar{g})_{\varphi^{-1}}$ as $\bar{G}$-graded $(\Od R,\Od T)$-bimodules. Thus, in particular,
\[(\Od R)(\bar{g})_{\varphi^{-1}}\simeq \Od [Rx^{-1}] \simeq \Od[x^{-1}T]\]
for some $x\in G$ such that $\varphi(u)=xux^{-1}$ for any $u\in R$. But then $\bar{x}=\bar{g}$, and moreover, since
\[Aj\simeq (Al)(\bar{g}^{-1})_{\varphi}\simeq Al(\bar{g}^{-1})x\] as $\bar{G}$-graded $(A,\Od R)$-bimodules, we get
\[Ajx^{-1}\simeq (Al)(\bar{g}^{-1}),\] and thus we obtain ${}^xR_{\epsilon}=T_{\rho}$. In conclusion, there is $x\in N_G(R_{\epsilon}, T_{\rho})$ such that $\bar{x}=\bar{g}$ and $\varphi=\varphi_x.$
\end{proof}

\section{Group-graded basic Morita equivalences}  \label{s:grMorita}

In the final section we prove the statements of Theorem \ref{thmB}. We show that if there is a group-graded basic Morita equivalence between two block extensions, then their extended local categories are preserved. As a consequence, we deduce the important uniqueness statement \cite[Theorem 1.8]{KuPu} and \cite[Theorem 3.5]{PuZhIII} concerning extensions of nilpotent blocks.

We start with the following two properties of fusions, motivated by the fact that a basic Morita equivalence is the composition of an embedding of $P$-interior algebras and an equivalence given by tensoring with a Dade $P$-algebra.  We state them without proof, since \cite[Lemma 1.17]{KuPu} and \cite[Proposition 2.14]{PuLof} can be easily adapted to our graded context.

\begin{proposition}\label{lem117} Let $A$ be a $\bar{G}$-graded  $P$-interior algebra having an $\Od$-basis with homogeneous elements such that $PXP=P $ and $|P\cdot x|=|P|=|x\cdot P|$ for any $x\in X$. Let $S$
be a  Dade $P$-interior algebra.

Let $R_{\epsilon}, T_{\rho}$ be local pointed groups on the $1$-component  $B$ of $A$, let $\epsilon''$, $\rho''$ be the unique local points of $R$, $T$ on $S$, and let $\epsilon'$, $\rho'$ be the unique local points of $R$, $T$ on $S\otimes B$ such that $\Br_R^S(\epsilon'')\otimes \Br_R^{B}(\epsilon)\subset \Br_R^{S\otimes B}(\epsilon ')$ and $\Br_T^S(\rho'')\otimes \Br_T^{B}(\rho)\subset \Br_T^{S\otimes B}(\rho')$. Then we have the equality
\[F_{S\otimes A}^{\bar{G}}(R_{\epsilon'}, T_{\rho'})=F_A^{\bar{G}}(R_{\epsilon}, T_{\rho})\cap \left(F_S(R_{\epsilon''}, T_{\rho''})\times \bar{G}\right)\]
\end{proposition}

\begin{proposition}\label{prop214} Let $f:A\rightarrow A'$ be an embedding of $\bar{G}$-graded $P\simeq P'$-interior algebras. Let  $R_{\epsilon}$, $T_{\rho}$ be two local pointed groups on the first component $A_1$ and let  $R'_{\epsilon'}$, $T'_{\rho'}$ be local pointed groups on $B'$ which correspond under the embedding $f$ such that $f(\epsilon)\subset \epsilon'$ and $f(\rho)\subset \rho'$. Then there exist bijections
\[E_G^{\bar{G}}(R_{\epsilon}, T_{\rho})\rightarrow E_{G'}^{\bar{G}}(R'_{\epsilon'}, T'_{\rho'})\qquad \textrm{ and } \qquad F_A^{\bar{G}}(R_{\epsilon}, T_{\rho})\rightarrow F_{A'}^{\bar{G}}(R'_{\epsilon'}, T'_{\rho'}).\]
\end{proposition}

We can now state our main result.

\begin{thm} \label{t:main-fus} With the notations of \ref{subsec81} we assume that $A$ is $\bar{G}$-graded basic Morita equivalent to $A'$. Then $\mathcal{E}$ is equivalent to $\mathcal{E}'$.
\end{thm}

\begin{proof} By \ref{subsec82} (\ref{eq9}) there is a bijection between the objects of $\mathcal{E}$ and the objects of $\mathcal{E}'$ such that $R_{\epsilon}$ (which is included in $P_{\gamma}$) is mapped into $R'_{\epsilon'}$ (which is included in $P'_{\gamma'}$), with $R\simeq R'$.
Let $R_{\epsilon}, T_{\rho}\leq P_{\gamma}$, respectively $R'_{\epsilon'}, T'_{\rho'}\leq P'_{\gamma'}$ be local pointed groups which correspond under the above bijection.

Since the morphisms in $\mathcal{E}, \mathcal{E'}$ are pairs given  by compositions of ``inclusions" and isomorphisms it is enough to prove that we have a bijection between $E_G^{\bar{G}}(R_{\epsilon}, T_{\rho})$ and $E_{G'}^{\bar{G}}(R'_{\epsilon'}, T'_{\rho'})$ when $R\simeq T$.
By using Proposition \ref{lem117} and \ref{prop214}, the same arguments as in \cite[7.6.3]{PuLo} assure us that there is a bijection between
$F_A^{\bar{G}}(R_{\epsilon}, T_{\rho})$ and $F_{A'}^{\bar{G}}(R'_{\epsilon'}, T'_{\rho'})$ for any $R\simeq T$.  Next, by Proposition \ref{prop44}, for any $R\simeq T$ there is a bijection
from $F_A^{\bar{G}}(R_{\epsilon}, T_{\rho})$ to $E_G^{\bar{G}}(R_{\epsilon}, T_{\rho})$  and the same bijection exists for the pointed groups on $B'$ from $F_{A'}^{\bar{G'}}(R'_{\epsilon'}, T'_{\rho'})$ to $E_{G'}^{\bar{G'}}(R'_{\epsilon'}, T'_{\rho'})$, hence our categories are equivalent.
\end{proof}

We have already mentioned that we may consider $(A,\bar G)$-fusions on $\bar G$-graded $\Od P$-interior algebras. It is easy to see that $(G,\bar G)$-fusions and the category $\mathcal{E}$ generalize to twisted group algebras of the form $\mathcal{O}_\alpha L$, where $\alpha\in Z^2(L,k^\times)$ (that is, to $k^\times$-groups in Puig's terminology), since such algebras are still $P$-interior. Thus, we obtain the next corollary, which is \cite[Theorem 1.8]{KuPu} and \cite[Theorem 3.5]{PuZhIII}, as a consequence of Theorems \ref{t:extnilp} and \ref{t:main-fus}.

\begin{corollary} \label{c:cornilp} Assume that the block $B$ is nilpotent. Then the categories $\mathcal{E}_{(b,H,G)}$ and $\mathcal{E}_{(1,\tau(Q),L)}$ are equivalent.
\end{corollary}

Finally, we show that the twisted group algebras $\mathcal{O}_\alpha L$ defined in Section \ref{sec2} is invariant under graded basic Morita equivalences, thus proving statements b) and c) of Theorem \ref{thmB}. We denote by $\mathcal{O}_{\alpha'}L'$ the twisted group algebra obtained from $A'$.

\begin{corollary} \label{c:corLunique} Assume that $A$ is $\bar{G}$-graded basic Morita equivalent to $A'$. Then there is an isomorphism $L\simeq L'$ as extensions of $Z(Q)\simeq Z(Q')$ by $E=N_G(Q_\delta)/C_H(Q)\simeq  N_{G'}(Q'_{\delta'})/C_{H'}(Q')$, and with this identification, $[\alpha]=[\alpha']$ in $H^2(\tilde E, k^\times)$.
\end{corollary}

\begin{proof} By Theorem \ref{t:main-fus} we have defect pointed groups $Q'_{\delta'}\le P'_{\gamma'}$ corresponding to $Q_{\delta}\le P_{\gamma}$ such that $P\simeq P'$. By \cite[Theorem 1.2]{CoMaTo}, we have the isomorphism $\tilde E=\tilde E_G^{\bar G}(Q_\delta)\simeq \tilde E_{G'}^{\bar G}(Q'_{\delta'})$, and  there is an $\tilde E$-graded basic Morita equivalence between $kN_{G}(Q_{\delta})b_{\delta}$ and $kN_{G'}(Q'_{\delta'})b'_{\delta'}$. Under this equivalence, the unique simple  $kQC_H(Q)b_\delta$-module $\bar V$ corresponds to the unique simple  $kQC_{H'}(Q')b'_{\delta'}$-module $\bar V'$ (see  \ref{d:cocycle-beta}). By \cite[Theorem 5.1.8]{Ma}, the Clifford extensions of $\bar V$ and $\bar V'$ are isomorphic, hence $[\alpha]=[\alpha']$ in $H^2(\tilde E, k^\times)$.

By Theorem \ref{t:main-fus} and Corollary \ref{c:cornilp} we have the equivalences
\[\mathcal{E}_{(b_\delta,QC_H(Q),N_G(Q_\delta))} \simeq \mathcal{E}_{(b'_{\delta'},Q'C_{H'}(Q'),N_{G'}(Q'_{\delta'}))} \]\[
\mathcal{E}_{(1,\tau(Q),L)} \simeq \mathcal{E}_{(1,\tau'(Q'),L')} \]
of categories.  Since in particular, \[L/\tau(Z(Q))\simeq N_G^{\bar G}(Q_\delta)/C_H(Q)\simeq  N_{G'}^{\bar G}(Q'_{\delta'})/C_{H'}(Q)\simeq L'/\tau(Z(Q')),\]
we may now use the argument in the final part of the proof of \cite[Theorem 3.5, p. 820]{PuZhIII} and \cite[Lemma 3.6]{PuZhIII} to deduce the isomorphism $L\simeq L'$ of group extensions.
\end{proof}

\begin{remark} Under the assumptions of Theorem \ref{t:main-fus}, from the $\tilde E$-graded basic Morita equivalence between $kN_{G}(Q_{\delta})b_{\delta}\simeq kN_{G'}(Q'_{\delta'})b'_{\delta'}$, \cite[Theorem 3.1]{CoMa} and Corollary \ref{c:cornormaldef}, we immediately get an embedding $k_{\alpha'} L'\to T\otimes k_\alpha L$ of $\tilde E$-graded $P$-interior algebras, for some Dade $P$-algebra $T$. Then, by \cite[Lemma 4.5]{Pu1}, we have that $T$ is similar to $k$.
\end{remark}

\bigskip\noindent{\bf Funding.} \ The last author was partially supported by a grant of Ministery of Research and Innovation, CNCS - UEFISCDI, project number PN-III-P1-1.1-TE-2016-0124, within PNCDI III.

\bigskip\noindent{\bf Acknowledgements.} \ The last two authors would like to thank Charles Eaton and Michael Livesey for valuable discussions and their kind hospitality during the visit to the School of Mathematics, University of Manchester.


\end{document}